\documentclass{amsart}
\usepackage{hyperref}
\usepackage{amsmath,amsthm}
\usepackage{amsfonts}
\usepackage{tikz}
\usepackage{mathrsfs}
\hypersetup{colorlinks,linkcolor=black,urlcolor=blue, citecolor=black}

\newcommand \eps {\varepsilon}

\newcommand \B {\mathcal B}
\newcommand \C {\mathcal C}

\renewcommand \H {\mathcal H}

\newcommand \E {\mathbb E}
\newcommand \R {\mathbb R}
\newcommand \T {\mathbb T}

\newcommand \N {\mathbb N}

\newcommand \Z {\mathbb Z}

\newcommand \m {\mathfrak m}
\newcommand \comp{\mathbb C}
\newcommand \inv {^{-1}}
\newcommand \nin{\not\in}

\newcommand \symdif{\bigtriangleup}
\newcommand \directint {\overset{\oplus}{\int}}

\newcommand* \sys {(X,\B,m)}

\newcommand* \nsys {(Y,\C,\nu)}
\newcommand* \Lp[2] {L_{#1}{#2}}
\newcommand* \bb [1]{\left({#1}\right)}
\newcommand* \binr[2]{\left<{#1},{#2}\right>}
\newcommand* \bset[1] {\left\{{#1}\right\}}
\newcommand* \indic [1]{\mathbf 1_{#1}}
\newcommand* \norm[2]{\left|\left|{#1}\right|\right|_{#2}}
\newcommand* \abs[1]{\left|{#1}\right|}
\newcommand* \integrate[4]{\int_{#1}^{#2}{#3}d{#4}}
\newcommand* \limit [2]{\underset{{#1}\rightarrow{#2}}{\lim}}
\newcommand* \bunion[3]{\overset {#3}{\underset{{#1} = {#2}}{\bigcup}}}
\newcommand* \bunions[2]{{\underset{{#1} \in {#2}}{\bigcup}}}
\newcommand* \sumit[3]{\overset {#3}{\underset{{#1} = {#2}}{\sum}}}
\newcommand* \sums[2]{\underset{{#1} \in {#2}}{\sum}}
\newcommand* \prodit[3]{\overset {#3}{\underset{{#1} = {#2}}{\prod}}}
\newcommand* \prods[2]{\underset{{#1} \in {#2}}{\prod}}
\newcommand* \bintersect[3]{\overset {#3}{\underset{{#1} = {#2}}{\bigcap}}}
\newcommand* \bintersects[2]{{\underset{{#1} \in {#2}}{\bigcap}}}

\newtheorem{thm}{Theorem}[section]
\newtheorem{lem}[thm]{Lemma}

\newtheorem{prop}[thm]{Proposition}
\newtheorem{cor}[thm]{Corollary}
\newtheorem{defn}[thm]{Definition}
\newtheorem{examp}[thm]{Example}

\newtheorem{rmk}[thm]{Remark}

\newtheorem*{theorem-non}{Theorem}
\newtheorem*{state-non}{Statement}

\begin{document}
\title{Ergodic Multiplier Properties}
\author{Adi Gl\"ucksam}
\today
\maketitle
\begin{abstract}
In this article we will see some properties that guarantee that a product of an ergodic non-singular action and a probability preserving ergodic action is also an ergodic action. We will start by proving 'The multiplier theorem' for locally compact abelian groups. Then we will show that for certain locally compact Polish groups (Moore groups, and minimally weakly mixing groups), a non-singular $G$ action is weakly mixing if and only if any finite dimensional $G$-invariant subspace of $L_\infty(X,\mathcal B,m)$ is trivial. Finally, we will show that the Gaussian action associated to the infinite dimensional irreducible representation of the continuous Heisenberg group, $H_3(\mathbb R)$, is weakly mixing but not mildly mixing.
\end{abstract}
\section{Introduction}
\begin{defn}\label{pptDef}
Let $\sys,\nsys$ be probability spaces, and let $A\in\B, A'\in\C$. A map $f:A\rightarrow A'$ is called a {\bf probability preserving transformation} if it is measurable and for every $C\in\C\cap A'$
$$
\nu(C)=m(f\inv(C))
$$
We denote by $PPT(X)$ the group of all invertible measure preserving transformations from $\sys$ to itself.
\end{defn}
\begin{defn}\label{nstDef}
Let $\sys,\nsys$ be probability spaces, and let $A\in\B, A'\in\C$. A map $f:A\rightarrow A'$ is called a {\bf non-singular transformation} if it is measurable and for every $C\in\C\cap A'$
$$
\nu(C)=0 \iff m(f\inv(C))=0
$$
We denote by $NST(X)$ the group of all invertible non-singular transformations from $\sys$ to itself. 
\end{defn}
\begin{defn}\label{PolishDef}
A {\bf Polish space} is a separable completely metrizable topological space; that is, a space homeomorphic to a complete metric space that has a countable dense subset.
\end{defn}
Let $\sys$ be a standard probability space, $T$ a non-singular transformation. Define the isometry $U=U_T$ on $\Lp2{\sys}$ by:
$$
(U_Tf)(x)=\sqrt{\frac{d\bb{m\circ T}}{dm}}\cdot f(Tx)
$$
This isometry, called the Koopman isometry, is well defined, since $T$ is a non-singular action, then the Radon-Nikodym derivative is a non-negative function in $\Lp1{\sys}$. In addition, it is a positive isometry in the sense that for every $f\in\Lp2{\sys}$ such that $f\ge 0$
$$
U_Tf=\sqrt{\frac{d\bb{m\circ T}}{dm}}\cdot f(Tx)\ge 0
$$
since the Radon-Nikodym derivative is non-negative and so was $f$. Similarly for every $U$, a positive isometry of $\Lp2{\sys}$, there exists a non-singular transformation $T$ such that $U=U_T$. This result is due to Lamperti, a proof can be found in \cite{Royden1968} (page 333).\\
The group of positive isometries is a subgroup of the group of operators defined on $\Lp2{\sys}$. Endow the group of positive isometries with the weak operator topology. The map $T\mapsto U_T$ is a bijection. The group of non-singular transformations endowed with the topology generated by the pull-back of the topology of positive isometries is a separable Polish group.
\begin{defn}
A {\bf probability preserving action} is a continuous homomorphism $T:G\rightarrow PPT(X)$.
\end{defn}
\begin{defn}
A {\bf non-singular action} is a continuous homomorphism $T:G\rightarrow NST(X)$.
\end{defn}
\begin{defn}
$T$ is called an {\bf ergodic action} if the following condition holds: if for every $g\in G$, $m(T_gA\bigtriangleup A) = 0$, then either $m(A)=0$ or $m(X\setminus A)=0$.
\end{defn}
\begin{prop}\label{denseSubgroupErgProp}
Let $G$ be a locally compact Polish group, $\sys$ be a standard probability space, $T:G\rightarrow NST(X)$ be an ergodic action, and let $G_0\subseteq G$ be a dense subgroup. Then the action of $G_0$ on $\sys$ is also ergodic.
\end{prop}
\begin{proof} Let $A\in\B$ be a $G_0$-invariant subset and let $g\in G$. Denote by $U$, the induced Koopman representation. There exists a sequence $\bset{g_n}\subseteq G_0$ such that $g_n\rightarrow g$, therefore by continuity $U_{g_n}\rightarrow U_g$, and in particular for every function $f$ measurable $f\circ T_{g_n}$ converges to $f\circ T_g$ in measure. Specifically for $f=\indic A$, we have $\indic A\circ T_{g_n}\overset{m}\rightarrow \indic A\circ T_g$. Let there be $\eps>0$, then there exists $N$ such that for every $n>N$:
$$
m\bb{T_{g_n}A\symdif T_gA}<\eps
$$
We will show $A$ is a $G$-invariant set.
$$
m\bb{A\symdif T_gA}\le m\bb{A\symdif T_{g_n}A}+m\bb{T_{g_n}A\symdif T_gA}<\eps
$$
by the triangle inequality.
This is true for every $\eps$ and therefore $m\bb{A\symdif T_gA}=0$, which means $A$ is a $G$-invariant set, and by ergodicity of $G$ it is trivial. Since every $G_0$-invariant set is trivial we conclude the action of $G_0$ on $\sys$ is ergodic.
\end{proof}
\begin{defn}
A non-singular action of $G$ on a standard space is called {\bf weakly mixing} if for any probability preserving ergodic action of $G$, the product action is also ergodic.
\end{defn}
\begin{rmk}\label{tomRmk}
Let $T$ be a non-singular weakly mixing action of $\sys$, let $S$ be a probability preserving action on $\nsys$ such that $T\times S\times S$ is ergodic. Then for every $k\in\N$, $T\times \underbrace{S\times\dots\times S}_{k\text{ times }}$ defined on the product space, is weakly mixing.\\
To prove this, use the fact that for every probability preserving action, $R$ on $(Z,\mathcal D,\mu)$, $(T\times \underbrace{S\times\dots\times S}_{k\text{ times }})\times R=T\times(\underbrace{S\times\dots\times S}_{k\text{ times }}\times R)$ and the fact that $S$ is weakly mixing, and therefore $S\times R$ is also an ergodic probability preserving action.
\end{rmk}
\begin{defn}\label{reducedKoopmanDef}
Let $T$ be a probability preserving action of a locally compact Polish group, acting on a standard probability space $\sys$. The {\bf reduced Koopman representation} is the Koopman representation reduced to the subspace $\Lp2{\sys}_0:=\bset{f\in\Lp2{\sys};~\int fdm=0}$.
\end{defn}
\begin{theorem-non}[The weak mixing theorem for $\Z$ actions]\label{WeakMixZThm}
For probability preserving $\mathbb Z$ actions the following conditions are equivalent:
\begin{enumerate}
\item The action is weakly mixing.
\item Any finite dimensional invariant subspace of $\Lp2{\sys}$ is trivial.
\item The product action is ergodic.
\item For every probability preserving ergodic action, the product action is ergodic.
\item The spectral type of the reduced Koopman representation is non-atomic.
\end{enumerate}
\end{theorem-non}
This theorem was originally proved by Koopman and Von Neumann in 1932 in their article "Dynamical systems of continuous spectra" (see\cite{Koopman1932}). A simple proof can be found in Peterson's book (see \cite{Petersen1983} page 65, theorem 6.1).\\
We will soon see there is a generalization of this theorem for probability preserving actions of locally compact Polish groups.\\

Another result for $\Z$ actions is 'the ergodic multiplier theorem'. This theorem, originally proved by M.Keane, spectrally characterizes when the product of a non-singular action and a probability preserving action is ergodic. A proof can be found in Aaronson's book (see \cite{Aaronson1997} pages 81-82). In this article, we will extend 'the ergodic multiplier theorem' for general locally compact abelian groups.\\

A natural next step was to study similar properties for more general groups. The first progress in this direction, originally proved by Dye and later improved by others, concerned unitary representations of locally-compact Polish groups. This means it can only be applied to probability preserving actions. To understand these results one shall need the following definitions:
\begin{defn}\label{C_bTopDef}
Denote $C_b(G)$ the set of bounded continuous functions defined on $G$. We say a sequence converges in the {\bf strong topology} if it converges in the supremum norm. We say a sequence $\bset{x_n}$ converges to $x$ in the {\bf weak topology} if for every linear functional $f$, $f(x_n)\rightarrow f(x)$.
\end{defn}
\begin{defn}\label{APandWAPDef}
A function $f\in C_b(G)$ is called {\bf almost periodic} if the orbit $\bset{f_g; g\in G}$ is pre-compact in $C_b(G)$ endowed with the strong topology. $f\in C_b(G)$ is called {\bf weakly almost periodic} if the orbit $\bset{f_g; g\in G}$ is pre-compact in $C_b(G)$ endowed with the weak topology.\\
Denote by $WAP(G)$ the algebra of weakly almost periodic functions, and by $AP(G)$ the algebra of almost periodic functions.
\end{defn}
\begin{examp} Let $\sys$ be a standard probability space. Given a unitary representation $\pi$ of $G$ on $\H:=\Lp2{\sys}$, for every two functions $f_1,f_2\in\H$ define the function $\varphi_{f_1,f_2} = \varphi:G\rightarrow\mathbb C$ by
$$
\varphi(g) = \binr{\pi(g)f_1}{f_2}
$$
Then $\varphi_{f_1,f_2}$ is weakly almost periodic. If in addition $\m(|\varphi_{f_1,f_2}|)=0$, then $\varphi_{f_1,f_2}$ is called a {\bf flight function}. For more details see \cite{Glasner2003}.
\end{examp}
\begin{defn}\label{invMeanDef}
Let $\mathfrak A$ be a liner space of bounded functions on $G$, which is closed under conjugations and contains the constant functions. A linear functional defined on $\mathfrak A$,  $\m$, is called a {\bf mean} if:
\begin{enumerate}
\item[$\bullet$] For every $f\in\mathfrak A$, $\m(\overline f)=\overline{\m(f)}$.
\item[$\bullet$] If $f\ge 0$, then $\m(f)\ge 0$.
\item[$\bullet$] $\m(\mathbf 1)=1$.
\end{enumerate}
For every function $f\in\mathfrak A$, denote by $f_g$ the function $f_g(h)=f(gh)$. $\m$ is called a {\bf left $G$-invariant mean} if it is a mean and in addition $m(f_g)=m(f)$ for every $f\in\mathfrak A$ and $g\in G$.
There is a similar notion for right $G$-invariant mean.
\end{defn}
\begin{thm}
The algebra $WAP(G)$ admits a unique left and right $G$-invariant mean, denoted by $\m$.
\end{thm}
\begin{proof} see \cite{Ellis1989} theorem III.A.2.
\end{proof}
By using the invariant mean on weakly almost periodic functions, similar equivalences to the ones presented for $\Z$ actions, can be deduced for representations of locally compact Polish groups.
\begin{defn}\label{weakMixRepDef}
A continuous unitary representation of a group $G$ on a Hilbert space $\H$ is called {\bf weakly mixing} if for every $f_1,f_2\in\H$, $\m\bb{\abs{\binr{\pi(g)f_1}{f_2}}} =0$.\\
It is called {\bf ergodic} if every $f_1,f_2\in\H$, $\m\bb{\binr{\pi(g)f_1}{f_2}} =0$.
\end{defn}
\begin{thm}[Corollary from Bergelson \& Rosenblatt 1988 \cite{Bergelson1988}] \label{bergelsonRosenblattThm}
Let $\sys$ be a standard probability space, $G$ a locally compact Polish group, $T:G\rightarrow PPT(X)$. Let $\pi$ be the reduced Koopman representation of $G$ induced by $T$. Then the following are equivalent:
\begin{enumerate}
\item $\pi$ is weakly mixing.
\item For every $f_1,f_2\in \Lp2{\sys}_0$:
$$
\m\bb{\abs{\int \bb{\pi(g)f_1}f_2dm}}=0
$$
\item $\pi$ contains no non-zero finite dimensional sub-representations.
\item The unitary representation $\pi\times\pi:G\rightarrow AUT(\H\times\H)$ is ergodic.
\item For every unitary representation $\sigma:G\rightarrow AUT(\mathcal K)$, on a Hilbert space $\mathcal K$, the unitary representation $\pi\times\sigma:G\rightarrow AUT(\H\times\mathcal K)$ is ergodic.
\end{enumerate}
\end{thm}
\begin{rmk}[Bergelson \& Gorodnik \cite{Bergelson2004}]\label{bergelsonCor}
Let $G$ be a locally compact Polish group, $T:G\rightarrow PPT(X)$ a probability preserving action of $G$ on a standard probability space $\sys$. $T$ is weakly mixing if and only if it is reduced Koopman representation is weakly mixing if and only if there is no finite dimensional $T$-invariant subspaces in $\Lp2{\sys}_0$.
\end{rmk}
\begin{prop}\label{weak&measureProp}
Let $G$ be a locally compact Polish group, $T$ a non-singular action on a standard probability space $\sys$. Denote by $U$ the Koopman representation, defined by:
$$
U_gf = \sqrt{\frac{dm_g}{dm}}\cdot f\circ T_g=\sqrt{\frac{d(m\circ T_g)}{dm}}\cdot f\circ T_g
$$
If there is no absolutely continuous invariant probability measure, then $U$ is weakly mixing.
\end{prop}
\begin{proof}
Since $U$ is a continuous unitary representation, if it is not weakly mixing, by corollary \ref{bergelsonCor}, there exists a finite dimensional non-trivial subspace in $\Lp2{\sys}$. Let $V = cls\bb{span\bset{v_1,v_2,\dots,v_d}}$ be the invariant subspace. Choose $f_1,\dots,f_d$ an orthogonal basis for $V$, such that $\norm{f_j}{\Lp2{\sys}}=1$. Note that since for every $g\in G$, $U_g$ is unitary, then for every $f\in\Lp2{\sys}$, $\norm{U_gf}{\Lp2{\sys}}^2=\norm{f}{\Lp2{\sys}}^2$, since:
$$
\norm{U_gf}2^2=\int_X\abs{U_gf}^2dm=\int_X\abs{\sqrt{\frac{dm_g}{dm}}f\circ T_g}^2dm=
$$
$$
=\int_X\abs{f\circ T_g}^2\cdot \frac{dm_g}{dm}dm=\int_X\abs{f}^2dm=\norm{f}2^2
$$
For every $1\le j\le d$ and for every $g\in G$ there exists $a_1^j(g),a_2^j(g)\dots,a_d^j(g)\in\R$ such that $U_gf_j=\sumit i 1 d a_i^j(g)f_i$, since $U_gf_j\in V$. Define the matrix
$$
A_g:=\begin{pmatrix}
a_1^1(g) & a_2^1(g) & \dots & a_n^1(g)\\ a_1^2(g) & a_2^2(g) & \dots & a_n^2(g)\\ \vdots &\vdots &\dots &\vdots \\ a_1^n(g) & a_2^n(g) & \dots & a_n^n(g)
\end{pmatrix}
$$
Since $U_g$ is linear, for every $\sumit j 1 d a_jf_j=f\in V$, if we denote by $\psi(\overline a)=\sumit j 1 d a_jf_j$, then:
$$
(U_gf)(x)=\sumit j 1 d a_j\cdot(U_gf_j)=\psi\bb{A_g^T\cdot \begin{pmatrix}a_1\\\vdots \\a_d\end{pmatrix}}
$$
We will show that for every $x\in \R^d$
$$
\norm {A_gx}2:=\norm{x}2
$$
Define $\pi:V\rightarrow\R^d$ by
$$
\pi\bb{\sumit j 1 d a_jf_j}=	\begin{pmatrix}
					a_1\\a_2\\\vdots\\a_d
					\end{pmatrix}
$$
Then since $\bset{f_j}$ is an orthonormal basis, $\pi$ is an isometry:
$$
\norm{\pi\bb{\sumit j 1 d a_jf_j}}2^2=\norm{\begin{pmatrix}
					a_1\\a_2\\\vdots\\a_d
					\end{pmatrix}}2^2=\sumit j 1 d\abs{a_j}^2=\norm{\sumit j 1 d a_jf_j}{\Lp2{\sys}}^2
$$
Note that $\pi(U_gf)=A_g^T\cdot\pi(f)$ (it is clear for the basis functions $\bset {f_1,\dots,f_d}$ and for the rest it is true by linearity). Since $\pi$ is an isometry we conclude that for every $x\in\R^d$:
$$
\norm{A_g^Tx}2=\norm{(U_g\pi\inv(x))}{\Lp2{\sys}}=\norm{\pi\inv(x)}{\Lp2{\sys}}=\norm x2
$$
We conclude that $\norm{A_g^Tx}2=\norm x2$, and therefore $\norm{A_gx}2=\norm x2$.\\
Define $F:X\rightarrow\R^d$ by
$$
F(x) = (f_1(x),f_2(x),\dots,f_d(x))
$$
and define the function $\phi:X\rightarrow \R$ by
$$
\phi(x) = \frac1d\sumit j 1 d\abs{f_j(x)}^2=\frac1d\norm{F(x)}2^2
$$
where $\norm*2$ denotes the Euclidean norm in $\R^d$.\\
Define the measure $d\mu = \phi\cdot dm$, then it is a probability measure, and for every set $B\in\B$:
$$
\mu(T_gB)=\underset X\int\indic{T_gB}d\mu(x)=\underset X\int\indic{T_gB}(x)\cdot\phi(x) dm(x)=
$$
$$
=\underset X\int\indic{T_gB}(x)\cdot\frac1d\sumit j 1 d\abs{f_j(x)}^2 dm(x)=\underset X\int\indic{B}(x)\cdot\frac1d\sumit j 1 d\abs{(U_gf_j)(x)}^2 dm(x)\overset{(\star)}=
$$
$$
=\underset X\int\indic{B}(x)\cdot\frac1d\norm{A_g\cdot F(x)}2^2 dm(x)\overset{(\star\star)}=\underset X\int\indic{B}(x)\cdot\frac1d\norm{F(x)}2^2 dm(x)=\mu(B)
$$
Where $(\star)$ is because:
$$
\norm{A_g\cdot F(x)}2^2=\sumit j 1 d\abs{\sumit i 1 d a_i^j(g)f_i}^2=\sumit j 1 d\abs{U_gf_j}^2
$$
and $(\star\star)$ is true since for every $v\in\R^d$, $\norm{A_gv}2=\norm v 2$, and specifically it is true for $v\in\R^d$ of the form $v=F(x)$ for some $x\in X$.
\\$\mu$ is a probability measure, it is action invariant, and it is absolutely continuous with respect to the original measure, by definition. We conclude that if there is no absolutely continuous, invariant, probability preserving measure, then $U$ is weakly mixing.
\end{proof}
Bergelson \& Rosenblatt's result is not enough to characterize weak mixing when it comes to non-singular actions. There are examples of non singular actions which are not weakly mixing, but no invariant probability measure exists, and by proposition \ref{weak&measureProp}, their reduced Koopman representation is weakly mixing. A non-singular adding machine (as described in \cite{Aaronson1997} pages 29-31) is conservative, ergodic, and has no absolutely continuous invariant measure, and therefore by proposition \ref{weak&measureProp} its Koopman representation is weakly mixing. Nevertheless, it is isomorphic to a group rotation, and therefore has $L_\infty$ eigenvalues, which means the action itself is not weakly mixing.\\

In this article, we will use representation theory to find ergodic multiplier properties for non-singular actions of locally compact Polish Moore groups (see section 4). We will show that if $G$ is a locally compact Polish Moore group, then a non-singular action of $G$ is weakly mixing if and only if every finite dimensional $G$-invariant subspace of $\Lp\infty{\sys}$ is trivial. In addition, we will show some examples of such actions.\\

We will conclude the article with an action of the Heisenberg group, which is weakly mixing but not mildly mixing (see section 4.3).\\
\section{Spectral properties of non-singular actions of locally compact abelian groups}
\subsection{Preliminaries}
\begin{defn}\label{charDef}
Let $G$ be a locally compact abelian group. A {\bf character}, is a continuous group homomorphism between $G$ and $S^1$, usually denoted by $\chi$. The characters form a group, which is called "the Dual group of $G$", and denoted by $\hat G$. $\hat G$ is also a subspace of all the continuous functions from $G$ to $\comp$, endow $\hat G$ with the topology of uniform convergence on compact sets. By Pontryagin duality theorem, since $G$ is locally compact, then $\hat{\hat G} = G$. Finally, the group action of the dual group is given by point-wise multiplication, and the inverse of a character is its complex conjugate.
\end{defn}
\begin{defn}\label{eigenDef}
An {\bf eigenvalue of the action $T$} is a character, $\chi\in \hat G$, such that there exists a non-zero measurable function $f\in\Lp2{\sys}$ for which:
$$
f\circ T_g = \chi(g)\cdot f ~~ -m \text{ almost everywhere}
$$
for every $g\in G$. The function $f$ will be called an {\bf eigenfunction}.\\
Denote by $e(T)$ the set of all eigenvalues of the action $T$.
\end{defn}
\begin{defn}\label{spectralMeasureDef}
For a complex separable Hilbert space $\H$, let $\mathcal E$ denote the collection of orthogonal projections in $\H$. Let $(X,\mathcal B)$ be a measurable Borel space. A function $E:\mathcal B\rightarrow\mathcal E$ is called a {\bf spectral measure} if:
\begin{enumerate}
\item $E(X) = Id$
\item $E\left(\bigcup_{n = 1}^\infty B_n\right) = \sum_{n = 1}^\infty E(B_n)$, for any pairwise disjoint sets.\\
The equality above should be interpreted for every $h\in \H$ in the following sense:
$$
\left(E\left(\bigcup_{n = 1}^\infty B_n\right)\right)(h) = \sum_{n = 1}^\infty (E(B_n))(h)
$$
\end{enumerate}
\end{defn}
\begin{thm}[The spectral theorem (see \cite{Stone1932a}, \cite{Neumann1932}, \cite{Hellinger1909}, \& \cite{Nadkarni1998})]\label{spectralThm}
Let $U_g:\mathcal H\rightarrow\mathcal H$ be a continuous unitary representation of $G$. Then there exists a spectral measure $E$ on $\hat G$ such that for every $g\in G$:
$$
U_g = \int_{\hat G}\chi(-g)dE
$$
\end{thm}
As a result for every $\alpha,\beta\in\H$ there exists a measure $m_{\alpha,\beta}$ defined on $\hat G$ such that:
$$
\binr{U_g\alpha}\beta = \int_{\hat G}\chi(-g)dm_{\alpha,\beta}
$$
Moreover, there exists a measure, $\sigma_0$ on $\hat G$, such that for every $\alpha,\beta\in\H$, $m_{\alpha,\beta}$ is absolutely continuous with respect to $\sigma_0$. The measure $\sigma_0$ is defined uniquely up to equivalence of measures and is known as {\bf the spectral type of the the representation}.\\
\begin{cor}\label{sesquiCor}[The scalar spectral theorem]
Let $G$ be a locally compact abelian group, $T:G\rightarrow PPT(X)$ an ergodic action of $G$ defined on a standard probability space $\sys$. Denote by $\sigma_0$ the spectral type of the reduced Koopman representation induced by $T$. Then there exists a non-trivial sesquilinear map 
$$
h:\Lp2{\sys}_0\times \Lp2{\sys}_0\rightarrow L_1(\sigma_0)
$$
such that for every $\alpha,\beta\in\Lp2{\sys}_0$, and for every $g\in G$:
$$
(h(U_g\alpha,\beta))(\chi) = \chi(-g)\cdot(h(\alpha,\beta))(\chi)
$$
where this equality is an $L_1(\sigma_0)$ equality.
\end{cor}
\begin{proof} By the 'the spectral theorem', for every $\alpha,\beta\in\Lp2\sys_0$ there exists a measure $m_{\alpha,\beta}$ such that
$$
\binr{U_g\alpha}\beta = \int_{\hat G} \chi(-g)dm_{\alpha,\beta}(\chi)
$$
Denote by $\sigma_0$ the spectral type of the representation $\{U_g\}_{g\in G}$, then $m_{\alpha,\beta}$ is absolutely continuous with respect to $\sigma_0$.\\
By Radon-Nikodym theorem there exists a function $f_{\alpha,\beta}\in L_1(\sigma_0)$ such that $\frac{dm_{\alpha,\beta}}{d\sigma_0} = f_{\alpha,\beta}\in L_1(\sigma_0)$. Define the map
$$
h(\alpha,\beta) = \frac{dm_{\alpha,\beta}}{d\sigma_0}
$$
First, it is indeed a sesquilinear map: let there be $a,b\in\mathbb C, \alpha,\beta,\gamma\in\Lp2{\sys}_0$, then-
$$
h(a\cdot\alpha+\beta,b\cdot\gamma) = \frac{dm_{a\cdot\alpha+\beta,b\cdot\gamma}}{d\sigma} \overset{(\star)}{=} a\cdot\overline b\cdot \frac{dm_{\alpha,\gamma}}{d\sigma}+\overline b\cdot\frac{dm_{\beta,\gamma}}{d\sigma}=
$$
$$
=a\cdot\overline b\cdot h(\alpha,\gamma)+\overline b\cdot h(\beta,\gamma)
$$
where $(\star)$ is true since for every $g\in G$ on one hand
$$
\binr{U_g(a\cdot\alpha+\beta)}{b\cdot\gamma} = \int_{\hat G}\chi(-g)dm_{a\cdot\alpha+\beta,b\cdot\gamma}
$$
but on the other hand
$$
\binr{U_g(a\cdot\alpha+\beta)}{b\cdot\gamma} = a\cdot\overline b\cdot\binr{U_g\alpha}\gamma+\overline b\cdot\binr{U_g\beta}\gamma = \int_{\hat G}\chi(-g)(a\cdot\overline b\cdot dm_{\alpha,\gamma}+\overline b\cdot dm_{\beta,\gamma})
$$
We conclude 
$$
dm_{a\cdot\alpha+\beta,b\cdot\gamma} = a\cdot\overline b\cdot dm_{\alpha,\gamma}+\overline b\cdot dm_{\beta,\gamma}
$$
Next, $h$ is non-trivial, otherwise all the operators $U_g$ are the identity, which is a contradiction to the ergodicity of the action. In addition, for every $g\in G$, and for every $\chi\in\hat G$:
$$
(h(U_g\alpha,\beta))(\chi) = \left(\frac{\chi(-g)dm_{\alpha,\beta}}{d\sigma}\right)(\chi) = \chi(-g)\cdot\left(\frac{dm_{\alpha,\beta}}{d\sigma}\right)(\chi) = \chi(-g)(h(\alpha,\beta))(\chi)
$$
\end{proof}
\begin{thm}[The eigenvalue theorem for locally compact abelian groups- originally proved by Schmidt (see \cite{Schmidt1982a}), a generalisation of the proof in \cite{Aaronson1997}]\label{eigenvalueThm}
Let $G$ be a locally compact abelian group, $T:G\rightarrow NST(X)$ a $G$-action defined on a standard probability space $(X,\mathcal B,m)$ such that $T$ is non-singular, and ergodic. Then:
\begin{enumerate}
\item[$\bullet$] The set $e(T)$ is a Borel set.
\item[$\bullet$] There exists $\psi:e(T)\times X\rightarrow S^1$ a measurable map such that for every $\chi\in e(T)$ and for every $g\in G$:
\begin{equation}\label{eigenEq}
\psi(\chi,T_gx) = \chi(g)\cdot\psi(\chi,x) \text{ for }m- \text{ almost every }x\in X
\end{equation}
\end{enumerate}
\end{thm}
\begin{proof} Let $E = \{f\in \Lp2{\sys};~ f\text{ is an eigenfunction of $T$}, |f| = 1\}$, and define the metric 
$$
d(f,h) = \int_X |f-h|^2dm
$$
then $E$ is a complete locally compact separable metric space with respect to this topology (as a subspace of $\Lp2{\sys}$, which is second countable since $X$ is a standard space). Denote by $\mathbb K\subseteq E$ the set of constant functions, and define the metric $\rho:E/\mathbb K\times E/\mathbb K\rightarrow\mathbb R$ by:
$$
\rho(f\mathbb K,h\mathbb K) = \inf_{c\in\mathbb K}d(f,c\cdot h)
$$
Then $E/\mathbb K$ is also a complete separable locally compact metric space.\\
Define the function $P:E/\mathbb K\rightarrow e(T)$ by
$$
P(f\mathbb K) = \overline f\cdot \bb{f\circ T}
$$
$P$ is well defined since every eigenfunction has a unique eigenvalue, it is continuous as an elementary function, one to one, and onto. Then $e(T)$ is a Borel set as an image of a complete separable locally compact metric space under a continuous one to one map.\\\\
Let $\{h_n;n\in\mathbb N\}$ be an orthonormal basis for $\Lp2{\sys}$, and define the following sets:
$$
K_n = \left\{f\in E;~\forall 1\le k\le n-1.~\binr{f}{h_k} = 0 \text{ and } \binr{f}{h_n}\neq 0\right\}
$$
Then by definition for every $m\neq n.~K_n\cap K_m = \emptyset$ and $\biguplus_{n = 1}^\infty K_n\cup \{0\} = \Lp2{\sys}$. In addition, $K_n$ are Borel sets as intersection of open and closed sets (inner product is a continuous function). Define the function $c:E\rightarrow\mathbb K$ by
$$
c(f) = \sumit n 1\infty\frac{\overline{\binr{f}{h_n}}}{\abs{\binr{f}{h_n}}}\cdot\indic{K_n}(f)
$$
Note that $c$ is well defined. In addition, $c$ is indeed measurable as inner products are continuous and $K_n$ are measurable sets. Define $M:E\rightarrow E$ by $M(f) = c(f)\cdot f$. Let $f\in K_n$, then for every $a\in\T$:
$$
M(a\cdot f) = c(a\cdot f)\cdot(a\cdot f) = \frac{\overline{\binr{a\cdot f}{h_n}}}{\abs{\binr{a\cdot f}{h_n}}}\cdot (a\cdot f)  =
$$
$$
= \frac{a\cdot\overline a}{\abs a}\cdot\frac{\overline{\binr{f}{h_n}}}{\abs{\binr{f}{h_n}}}\cdot f =\abs a\cdot \frac{\overline{\binr{f}{h_n}}}{\abs{\binr{f}{h_n}}}\cdot f = \frac{\overline{\binr{f}{h_n}}}{\abs{\binr{f}{h_n}}}\cdot f  = M(f)
$$
Define the function $N:E/\mathbb K\rightarrow E$ by
$$
N(f\mathbb K) = M(f)
$$
since for every $a\in\mathbb T$, $M(af) = M(f)$, and the function is well defined. Finally, define $\psi:e(T)\times X\rightarrow\mathbb C$ by
$$
\psi(\chi,x) = (N(P^{-1}(\chi)))(x)
$$
First of all, we will show that for every $\chi\in\hat G$ the function $\psi(\chi,*)$ is measurable as a function of $x$. Note that $P\inv(\chi)$ gives us an eigenfunction, which is clearly measurable. Now, $N(P\inv(\chi))$ is a measurable function as a limit of such functions. In addition, $\psi$ is well defined and for every $\chi\in e(T)$, and for every $g\in G$:
$$
\psi(\chi,T_gx) = (N(P^{-1}(\chi)))(T_gx) = (N(f\mathbb K))(T_gx) = (M(f))(T_gx) = 
$$
$$
c(f)\cdot f(T_gx) = c(f)\cdot f(x)\cdot \chi(g) = \chi(g)\cdot (M(f))(x) = \chi(g)\cdot\psi(\chi,x)
$$
It is left to show that for every $x\in X$ the function $\psi(*,x)$ is a Borel map with respect to $\chi$. First we will show that $P\inv$ is indeed a Borel map. Let $U$ be a Borel set, then $P(U)$ is an analytic set as a continuous image of a Borel set in a Polish space. In addition, $P$ is an injective therefore $P(U^c) = P(U)^c$, and $P(U^c)$ is also an analytic set, since $U^c$ is a Borel set. We conclude $P\inv$ is a Borel map. Next, for every $n\in\N$ the set $K_n$ is a Borel set, therefore for every $n\in N$ the function $\frac{\overline{\binr{f}{h_n}}}{\abs{\binr{f}{h_n}}}\cdot f\indic{K_n}(f)$ is a Borel function, as a multiplication of two Borel functions (note that since $x\in X$ is fixed, $f(x)$ is a constant number). Since sum of Borel functions, and point-wise limit of Borel functions is a Borel function, then so is the function $N$. Overall the function $\psi(*,x)$ is a Borel function as a composition of Borel functions.
\end{proof}
\begin{rmk}
It would be interesting to have a generalized version of the eigenvalue theorem for Moore groups (see section 4).
\end{rmk}
\begin{defn}\label{properlyErgodicDef}
A non-singular action of $G$ on a standard probability space is called {\bf properly ergodic} if it is ergodic and every orbit has measure zero, i.e for every $x\in X$, $m\bb{\bset{T_gx;\;\;g\in G}}=0$.
\end{defn}
\begin{rmk}\label{spectralMeasReducedKoopRmk}
We denote by $\chi_0$ the constant character. If $T$ is probability preserving and ergodic, then $\sigma_0(\bset{\chi_0})=0$, where $\sigma_0$ is the spectral type of the reduced Koopman representation induced by $T$. The reason for that, is that for every eigenvalue $\chi$, $\sigma_0(\bset{\chi})>0$, but since $T$ is ergodic, if $\chi_0$ is an eigenvalue, then there exists an invariant function. By ergodicity every invariant function is constant, and every constant function in $\Lp2{\nsys}_0$ is zero. We conclude that $\chi_0$ cannot be an eigenvalue for $T$ and therefore $\sigma_0(\bset{\chi_0})=0$.
\end{rmk}
\subsection{The ergodic multiplier theorem}
In this section we will prove the ergodic multiplier theorem for actions of locally compact Polish abelian groups.
\begin{thm}[The ergodic multiplier theorem]\label{multipThm}
Let $G$ be a locally compact Polish abelian group, let $\sys$ be a standard space, and let $T:G\rightarrow NST(X)$ be a non-singular, and properly ergodic $G$-action. Let $\nsys$ be a standard probability space, $S:G\rightarrow PPT(Y)$ an ergodic $G$-action, and denote by $\sigma_0$ the spectral type of reduced Koopman representation, induced by the action $S$. Then the following are equivalent:
\begin{enumerate}
\item $T\times S$ is ergodic.
\item $\sigma_0(e(T)) = 0$.
\end{enumerate}
\end{thm}
\begin{proof}
$(1)\Leftarrow(2)$: We will show that if $T\times S$ is not ergodic, then $\sigma_0(e(T))>0$. If $T\times S$ is not ergodic, then there exists $F:X\times Y\rightarrow\mathbb C$ non-constant invariant function. Without loss of generality, $F$ is bounded, since an indicator function on an invariant set would also suffice. Define the function $\phi:X\rightarrow \Lp2{\nsys}_0$ by:
$$
(\phi(x))(y) = F(x,y)-\int_{Y}F(x,y)d\nu(y)
$$
Note that $\phi(x)\in \Lp2{\nsys}_0$ and if $\phi$ is constant, then there exists $f\in \Lp2{\nsys}_0$ such that for $m$-almost every $x\in X$, $\phi(x) =f$, and $f(y) = F(x,y)$. But:
$$
f\circ S_g(y) = F(x,S_gy) = F(T_g(T_{-g}x),S_gy))=F(T_{-g}x,y) = f(y)
$$
$f$ is a $G$-invariant function. Since $S$ is an ergodic action, $f$ is constant and therefore so is $F$, but it was chosen to be a non-constant function. We conclude $\phi$ is not constant.\\
Denote by $Isom(\Lp2\nsys_0)$ the set of all invertible isometries of $\Lp2\nsys_0$, and by $U:G\rightarrow Isom(\Lp2\nsys_0)$ the mapping defined by $U_gf=f\circ S_{-g}$. Since $F$ is $T\times S$ invariant, it follows that for every $g\in G$, for $m$-almost every $x\in X$:
\begin{equation*}
\begin{split}
\phi(T_gx) =  F(T_gx,\cdot) =F(T_gx,S_g(S_{-g}(\cdot))) = \\
=F(x,S_{-g}(\cdot)) = \phi(x)\circ S_{-g}=U_g(\phi(x))
\end{split}
\end{equation*}
Using corollary \ref{sesquiCor}, there exists $h$ a non-trivial sesquilinear map \\
$h: \Lp2{\nsys}_0\times \Lp2{\nsys}_0\rightarrow L_1(\sigma_0)$ such that for every $f\in \Lp2{\nsys}_0$, every $g\in G$, and $m$ almost every $x\in X$:
\begin{equation}\label{sesquiInvEq}
\begin{split}
(h(\phi(T_gx),f))(\chi) = (h(\phi(x)\circ S_{-g},f))(\chi) =\\
=\chi(g)(h(\phi(x),f))(\chi)\;\;-\sigma_0 \text{ almost everywhere}
\end{split}
\end{equation}
Let $G_0\subseteq G$ be a dense countable subgroup, and denote by $T^0$ the action of $G_0$ on $\sys$. Denote by $\Xi_0\subseteq\hat G$ the set of elements such that equation (\ref{sesquiInvEq}) holds for every $g\in G_0$, it is of $\sigma_0$ full measure as countable intersection of such sets. Define for every $\chi\in\Xi_0$ and $f\in \Lp2{\nsys}_0$, the function $f_{\chi}: X\rightarrow \mathbb C$ by:
$$
f_{\chi}(x) = (h(\phi(x),f))(\chi)
$$
By equation (\ref{sesquiInvEq}), for every $g\in G_0$, and for $m$-almost every $x\in X$:
$$
(f_{\chi}\circ T_g)(x) = (h(\phi(T_gx),f))(\chi) = (h(\phi(x)\circ S_{-g},f))(\chi) =
$$
$$
=  \chi(g)(h(\phi(x),f))(\chi) = \chi(g)f_\chi(x)
$$
Therefore $f_\chi$ is an eigenfunction of the action of $G_0$. By proposition \ref{denseSubgroupErgProp}, the action of $G_0$ is ergodic as an action of a dense subgroup. Every eigenfunction of an ergodic action has constant absolute value, denote $|f_{\chi}| = c_{\chi}$.\\
To show that $\sigma_0(e(T^0))>0$ it is enough to show that for some $f\in \Lp2{\nsys}_0$ it is true that $\sigma_0\bb{\bset{\chi\in e(T^0);~c_{\chi}>0}}>0$.\\
If for every $f\in \Lp2{\nsys}_0$ the equality $\sigma_0\bb{\bset{\chi\in e(T^0);~c_{\chi}>0}}=0$ holds, then for every $f\in \Lp2{\nsys}_0$, $h(\phi(x),f) = 0$. By the Scalar Spectral Theorem, and the definition of $h$, for every $g\in G_0$ for $m$ almost every $x\in X$:
$$
\binr{\phi(x)}f = \int_{\hat G}h(\phi(x),f)(\chi)d\sigma_0(\chi) = 0
$$
And therefore, for almost every $x\in X$ the function $\phi(x)$ is constant with respect to Y. Define the function:
$$
E(x):=\int_{Y}F(x,u)d\nu(u)
$$
Then for $m\times\nu$-almost every $(x,y)\in X\times Y$:
$$
E(x):=\int_{Y}F(x,u)d\nu(u)=\int_{Y}(\phi(x))(u) d\nu(u)= (\phi(x))(y) =F(x,y) 
$$
In addition, $E$ is a $G_0$-invariant function, since $S$ is a probability preserving action, for $m$-almost every $x\in X$:
$$
(E\circ T_g)(x) = \int_{Y}F(T_gx,u)d\nu(u) =
$$
$$
= \int_{Y}F(T_gx,S_gu)d\nu(u)= \int_{Y}F(x,u)d\nu(u) = E(x)
$$
$T^0$ is ergodic by hypothesis, then $E$ is necessarily constant, but if $E$ is constant, then $F$ is constant, which is a contradiction.\\
We conclude that $\sigma_0(e(T^0))>0$. Finally, a proof as the one of proposition \ref{denseSubgroupErgProp}, shows that $e(T)=e(T^0)$, then $\sigma_0(e(T))=\sigma_0(e(T^0))>0$.\\

$(1)\Rightarrow(2)$: We will show that if $\sigma_0(e(T))>0$ then $T\times S$ is not ergodic. Since $T$ is a non-singular action, one may assume $m(X)<\infty$. If $\sigma_0(e(T))>0$, without loss of generality there exists $f\in \Lp2{\nsys}_0$ such that $\mu_f(e(T)) = \norm f{}^2$, where $\mu_f(A) = \binr{E(A)f}{f}$, and $E$ is the spectral measure that generates $\sigma_0$. This is since there exists $\mu<<\sigma_0$ such that $supp(\mu) = e(T)$ (for more information see \cite{Nadkarni1998} chapter 1).\\
Let $U_g$ denote the reduced Koopman representation of $S$, the action of $G$ in $\Lp2{\nsys}_0$. Denote $\mathcal H_f = \overline{span\{U_gf;~g\in G\}}$, and let $V:\mathcal H_f\rightarrow L_2(\hat G,\hat{\mathcal B},\mu_f)$ be the Hilbert space isometry defined as follows: for every $g\in G$ define
$$
(V(U_gf))(\chi) = \chi\bb{-g}\Rightarrow U_gf = V^{-1}\bb{\chi\bb{-g}}
$$
Continue defining $V$ as a linear transformation, then it is well defined on $\mathcal H_f$. Note that for every $g,h\in G$:
\begin{equation*}
\begin{split}
(V(U_g\circ U_h f))(\chi) =  (V(U_{g+h}f))(\chi) = \\
=\chi\bb{-\bb{g+h}} =\chi\bb{(-h)(-g)} = \chi\bb{-g}\cdot\chi\bb{-h}
\end{split}
\end{equation*}
Therefore
\begin{equation}\label{isometrySpanEq}
V^{-1}\bb{\chi\bb{-g}\cdot\chi\bb{-h}} = \bb{U_g\circ U_h f}(\chi)
\end{equation}
Since the set $span\bset{U_gf;~g\in G}$ is dense in $\mathcal H_f$, and for every $g\in G$ equality (\ref{isometrySpanEq}) holds, then for every $h\in\H_f$ and for every $g\in G$:
\begin{equation}\label{isometryEq}
V^{-1}(\chi(-g)\cdot h) = (U_g\circ V^{-1}(h))(\chi)
\end{equation}
We will show $V$ is an isometry. It is enough to show $V$ preserves the inner product for a dense subset in $\H_f$. Let there be $g\in G$, then:
$$
\norm{U_gf}{}^2 = \binr{U_gf}{U_gf} = \binr ff = \int_{Y}\norm f{}^2d\nu = \norm f{}^2 = \mu_f(\hat G)=
$$
$$
=\int_{\hat G}\mathbf 1 d\mu_f = \binr{\mathbf 1}{\mathbf 1} = \binr{\chi\bb{-g}}{\chi\bb{-g}} = \norm{\chi\bb{-g}}{}|^2 = \norm{V(U_gf)}{}^2
$$
By linearity of $V$ it is true for the linear span of $\{U_gf;~g\in G\}$ and by continuity of both $V$ and the inner product, for $\mathcal H_f$.\\
In addition, by 'the eigenvalue theorem' (theorem \ref{eigenvalueThm}) there exists a Borel function $\psi:e(T)\times X\rightarrow S^1$ such that for every $\chi\in e(T)$, every $g\in G$, and for $m$-almost every $x\in X$:
$$
\psi(\chi,T_gx) = \chi(g)\psi(\chi,x)
$$
Specifically, for every $\chi\in e(T)$ the transformation $f_\chi = \psi(\chi,*):X\rightarrow S^1$ is an eigenfunction of the eigenvalue $\chi$.\\
Now, let us look at the function $\phi: X\rightarrow \{\eta;~\eta:e(T)\rightarrow S^1\}$ defined by:
$$
(\phi(x))(\chi) = \psi(\chi,x)
$$
Remember that one can define the action of $G$ on $\hat G$ by $g(\chi) := \chi(g)$, this way any $g\in G$ becomes a function from $e(T)$ to $S^1$. By the properties of $\psi$ (equation (\ref{eigenEq})) for every $g\in G$ and $\chi\in e(T)$: 
\begin{equation*}
\begin{split}
(\phi(T_gx))(\chi) = \psi(\chi,T_gx) = \chi(g)\cdot\psi(\chi,x)=\\
 = (g(*)\cdot\phi(x))(\chi)\;\;m\text{- almost everywhere}
\end{split}
\end{equation*}
Define the function $F:X\times Y\rightarrow\mathbb C$ by:
$$
F(x,y) = V^{-1}(\phi(x))(y)
$$
First, note that $F$ is well defined almost everywhere, since for almost every $x\in X$ the function $\phi(x)$ is a function from $e(T)$ to $S^1$ and the domain of $V^{-1}$ is $\bset{f:\hat G\rightarrow\mathbb C}$. Next, by the definition of $V$, $V^{-1}(\phi(x))\in\mathcal H_f\subseteq \Lp2{\nsys}_0$. By using the definitions and equation (\ref{isometryEq}) for every $g\in G$:
$$
F(T_gx,S_gy) = U_g(V^{-1}(\phi(T_gx)))(y) = U_g(V^{-1}(g(*)\cdot\phi(x)))(y) = 
$$
$$
= U_g(U_{-g}\circ(V\inv(\phi(x))))(y) = V\inv(\phi(x))(y) = F(x,y)
$$
Which means $F(x,y)$ is $T\times S$ invariant.\\
It is left to show that $F$ is not constant. Since $V$ is an isometry, than so is $V\inv$ and therefore if we will show that $\psi$ is not constant with respect to $x$, then $F$ is not a constant function.\\
Let $G_0\subset G$ be a dense countable subgroup, and denote by $T^0$ the action of $G_0$ on $\sys$. Assume $\psi$ is constant with respect to $x$, and let there be $\chi\in e(T)$. Then there exists $x\in X$ such that for every $g\in G_0$, we have $\psi(\chi,x)=\psi(\chi,T_gx)$ and therefore
$$
\psi(\chi,x)=(\phi(x))(\chi)=(\phi(T_gx)(\chi)=\psi(\chi,T_gx)=\chi(g)\psi(\chi,x)
$$
We conclude that $\chi(g)=1$ for every $g\in G_0$, and by continuity of the characters for every $g\in G$. This is true for every $\chi\in e(T)$ and every $g\in G$, then $e(T)=\bset{\chi_0}$, but by remark \ref{spectralMeasReducedKoopRmk} $\sigma_0(\bset{\chi_0})=0$, which is a contradiction.
\end{proof}
\begin{rmk}
Note that given an eigenvalue theorem for a larger class of groups (for example Moore groups), one could easily extend the ergodic multiplier theorem for this class of group actions.
\end{rmk}
\section{Banach-Kronecker systems}
\subsection{Definitions \& Preliminaries}
\begin{defn}\label{topActSysDef}
Let $G$ be a locally compact Polish group, $X$ be a second countable compact topological space. A map $\pi:G\times X\rightarrow X$  is called a {\bf topological action} if $\pi$ has the following properties:
\begin{enumerate}
\item If $e$ is the identity element of $G$, then for every $x\in X$, $\pi(e,x)=x$.
\item $\pi$ is continuous with respect to both variables.
\item For every $g,h\in G$ and $x\in X$: $\pi(gh,x)=\pi(g,hx)$, where we have to following notation: $hx=\pi(h,x)$.
\end{enumerate}
\end{defn}
\begin{defn}\label{topDynSysDef}
Let $G$ be a locally compact Polish group, $X$ be a second countable compact topological space. If there exists a map, $\pi:G\times X\rightarrow X$, which is a topological action, then we say $(X,\pi,G)$ is a {\bf topological dynamical system}.
\end{defn}
\begin{defn}\label{equicontDef}
Let $(X,\pi,G)$ be a topological dynamical system. We say the system is {\bf equicontinuous} if for every $\eps>0$ there exists $\delta >0$ such that for every $x,y\in X$ if $d(x,y)<\delta$, then for every $g\in G$, $d(\pi(g,x),\pi(g,y))<\eps$.\\
\end{defn}
It is known that $X$ is a metric space, since it is compact and second countable. In addition, since $X$ is compact every two metrics on $X$ are equivalent, which means this property does not depend on the choice of metric.
\begin{defn}\label{transitiveDef}
Let $(X,\pi,G)$ be a topological dynamical system. We say $x_0\in X$ is a {\bf transitive point} if its orbit is dense in $X$, 
$$
\overline{O_G(x_0)}:=\overline{\bset{\pi\bb{g,x_0};~g\in G}}=X
$$
\end{defn}
\begin{defn}\label{minimalDef}
Let $(X,\pi,G)$ be a topological dynamical system. We say the system is {\bf minimal} if every $x\in X$ is a transitive point.
\end{defn}
\begin{lem}\label{existenceOfInvMetricLem}
Let $X$ be a compact second countable topological space, let $G$ be locally compact Polish group. If $(X,\pi,G)$ is an equicontinuous topological dynamical system, then there exists an invariant metric.
\end{lem}
\begin {proof}
Denote by $d$ the metric of $X$. Define the metric $D:X\times X\rightarrow\R_+$ by $D(x,y) = \sup_{g\in G}d(\pi(g,x),\pi(g,y))$. It is not difficult to see it is indeed a metric. We will show it generates the same topology. On one hand, by definition $D(x,y)\ge d(x,y)$ and therefore $B_D(x,r)\subseteq B_d(x,r)$. We will show that for every $\eps>0$ there exists $\delta>0$ such that $B_d(x,\delta)\subseteq B_D(x,\eps)$.\\
Let there be $\eps>0$, then by equicontinuity there exists $\delta>0$ such that for every $x,y\in X$ such that $d(x,y)<\delta$ for every $g\in G$, $d(\pi(g,x),\pi(g,y))<\frac\eps2$. We conclude that $B_d(x,\delta)\subseteq \overline{B_D\bb{x,\frac\eps2}}\subset B_d(x,\eps)$. Overall the topologies generated by the two metrics are the same.
\end {proof}
\begin{defn}\label{isometriesTopologyDef}
Let $X$ be a metric space, and let us look at the collection of invertible homeomorphisms defined on $X$. Define the following metric on the group of invertible homeomorphisms:
$$
\Delta(\varphi,\psi)=\underset{x\in X}\sup\;\;d(\varphi(x),\psi(x))+\underset{x\in X}\sup\;\;d(\varphi\inv(x),\psi\inv(x))
$$
The topology induced by this norm is called the {\bf compact open topology}.
\end{defn}
\begin{defn}\label{homogenousDef}
A dynamical system, $(K/H,\pi,G)$, is called {\bf a homogeneous system} if $K$ is a compact topological group, $H\le K$ is a closed subgroup, and $G$ is embedded in $K$ as a dense subgroup, acting on $K/H$ by left translations, meaning $\pi(g,k)=\tilde gk$, where $\tilde g$ is the embedding of $g$ in $K$.
\end{defn}
\begin{lem}\label{invErgMeasureLem}
Let $X$ be a compact metric space, $G$ a locally compact Polish group, $(X,\pi,G)$ a topological dynamical system, which is minimal and equicontinuous. Then there exists a homogeneous system $(K/H,\tilde\pi,G)$ such that $(X,\pi,G)$ is topologically isomorphic to a translation on $(K/H,\tilde\pi,G)$. Moreover, there exists an action invariant ergodic probability measure defined on $X$.
\end{lem}
\begin{proof}
Define $T:G\rightarrow InvHom(X)$ by $T_g(x)=\pi(g,x)$, then since the action is equicontinuous by Arzel\'a-Ascoli theorem the set $\bset {T_g;~g\in G}$ is a precompact subset of the set of homeomorphisms of $X$, endowed with the compact open topology. Denote by $K:=cls\bset {T_g;~g\in G}$, then $K$ is a compact subgroup of the group of homeomorphisms. Fix $x_0\in X$ and define $H:=cls\bset{T\in K;~Tx_0=x_0}$.\\
Define $\varphi:K/H\rightarrow X$ by $\varphi(TH)=Tx_0$.\\
{\bf $\varphi$ is continuous:} Let $T_nH,TH\in K/H$ such that $T_nH\underset{n\rightarrow\infty}\longrightarrow TH$ in the compact open topology. We will show that $\varphi(T_nH)=T_nx_0\underset{n\rightarrow\infty}\longrightarrow Tx_0=\varphi(TH)$. If  $T_nH\underset{n\rightarrow\infty}\longrightarrow TH$, then there exists a sequence $\bset{S_n}\subset H$ such that $T_nS_n\rightarrow T$ in the compact open topology. If $T_nx_0\not\rightarrow Tx_0$, then:
$$
0 = \limit n\infty\Delta(T_nS_n,T)\ge \underset{n\rightarrow\infty}\limsup\;\; d(T_nS_nx_0,Tx_0)>0
$$
which is a contradiction. We conclude that $\varphi$ is continuous. Now, the original action is minimal, and therefore the image of $\bset{T_g;~g\in G}/H$ is a dense subset of $X$, and since $K/H$ is closed and $\varphi$ is continuous, then the image of $K/H$ is closed and therefore $\varphi$ is onto $X$.\\
{\bf $\varphi$ is a bijection:} Let there be $T_1H,T_2H\in K/H$ such that $T_1H\neq T_2H$ and $\varphi(T_1H)=T_1x_0=T_2x_0=\varphi(T_2H)$, then $T_2\inv T_1\in H$, which is a contradiction. We conclude $\varphi$ is a bijection.\\
Since $K/H$ is compact, $\varphi$ is an isomorphism as a continuous bijection.\\
Define the topological action of $G$ on $K/H$ by left multiplication $\tilde\pi(g,TH)=(T_gT)H$. Then:
$$
\pi(g,\varphi(TH))=T_g(\varphi(TH))=T_g(Tx_0)=(T_gT)(x_0)=\varphi(\tilde\pi(g,TH))
$$
We conclude that $(X,\pi,G)$ is isomorphic as a topological dynamical system to the dynamical system $(K/H,\tilde\pi,G)$, a homogeneous system. Since $K/H$ is a compact set the projection of the Haar measure onto $K/H$ is an invariant probability measure for $T$.\\
Now, the set of invariant measures is a compact convex subset, by the Krein-Milman theorem, it is the convex hull of its extreme points. The action invariant ergodic measures are the extreme points of this set, specifically there exists an ergodic $G$-invariant measure.
\end{proof}
\begin{lem}\label{posMeasureBallLem}
Let $X$ be a compact metric space, $(X,G)$ be a topological dynamical system. Assume there exists an invariant metric $D$, and that the action is minimal. Then any non-singular measure is globally supported.
\end{lem}
\begin{proof}
Denote by $T$ the action of $G$ on $X$, and let $m$ be a non-singular measure. If $m$ is not globally supported, then there exists $x_0\in X$, and $\eps_0>0$ such that $m(B_D(x_0,\eps_0))=0$. Since the action is minimal the collection $\bset{B_D(T_gx_0,\eps_0)}$ is a cover for $X$. Since $X$ is compact there exists a finite sub-cover $\bset{B_D(T_{g_n}x_0,\eps_0)}_{n=1}^N$. Then:
$$
1=m(X)=m\bb{\bunion n 1 NB_D(T_{g_n}x_0,\eps_0)}\le
$$
$$
\le\sumit n 1 N m\bb{B_D(T_{g_n}x_0,\eps_0)}\overset{(\star)}=\sumit n 1 N m\bb{T_{g_n}\bb{B_D(x_0,\eps_0)}}
$$
Where $(\star)$ is because the metric is action invariant, and therefore- \\
$B_D(T_gx_0,\eps_0)=T_g\bb{B_D(x_0,\eps_0)}$. Now if $m\bb{B_D(x_0,\eps_0)}=0$, then for every $n$, $m\bb{T_{g_n}\bb{B_D(x_0,\eps_0)}}=0$, since $m$ is a non-singular measure, but then:
$$
1=m(X)\le\dots\le \sumit n 1 N m\bb{T_{g_n}\bb{B_D(x_0,\eps_0)}}=0
$$
Since this is not true, for every $x\in X$ and $\eps>0$, $m(B_D(x,\eps))>0$.
\end{proof}
\begin{rmk}
Note that this lemma is still true if $X$ is Polish and not compact. In this case the sub-cover is not finite, but countable, and the same proof holds.
\end{rmk}
\subsection{Banach-Kronecker systems are not weakly mixing}
\begin{defn}\label{BanachKronDef}
Let $X\subset \R^d$ be a closed bounded set, let $G$ be a locally compact Polish group. We say the linear topological dynamical system $(X,\pi,G)$ is a {\bf Banach-Kronecker system}, if $(X,\pi,G)$ is minimal and equicontinuous.
\end{defn}
\begin{rmk}
Note that according to lemma \ref{invErgMeasureLem}, there exists an action invariant measure, specifically there exist non-singular measures.
\end{rmk}
\begin{thm}\label{banachKronThm}
Let $(X,\pi,G)$ be a Banach-Kronecker system, denote by $T$ the topological action, meaning $T_g(x)=\pi(g,x)$. Then for any non-singular measure, $m\sim m\circ T$, the action is not weakly mixing.
\end{thm}
\begin{proof}
Let $m$ be a non-singular measure, if it is not $G$-ergodic, then clearly the system is not weakly mixing. We will show that if $m$ is ergodic, then there exists a probability invariant measure $P$ such that $(X\times X, \B\times\B,m\times P)$ is not ergodic.\\
Note that Banach-Kronecker systems are by definition minimal and equicontinuous, by lemma \ref{invErgMeasureLem} there exists a probability measure $P$ which is action invariant and ergodic. Define the following norm on linear homeomorphisms:
$$
\norm{A}{op}=\underset{\norm{x}{}=1}\sup\norm{Ax}{}
$$
Define the following metric on $X$, $d(x,y)=\norm{x-y}{}$. There exists a constant $C$ such that for every $g\in G$, $\norm{T_g}{op}<C$- if not then there exists a sequence $\bset{g_n}\subseteq G$ and $\bset{x_n}\subset X$ such that $\limit n\infty\norm{T_{g_n}x_n}{}=\infty$. But, norm is a continuous function, $X$ is compact, and therefore is has a global finite maximum, which is a contradiction.\\
For every $\eps>0$ define the set 
$$
D_\eps:=\bset{(x,y)\in X\times X;d(x,y)<\frac \eps{2C}}
$$
This set is an open set and therefore Borel measurable, and for every $g\in G$:
$$
d(T_g(x),T_g(y)) \le \norm{T_g}{op}\cdot d(x,y)<C\cdot d(x,y)<\eps
$$
Let $G_0\subseteq G$ be a dense countable subgroup in $G$. Define the set 
$$
O_\eps=\bset{(T_g(x),T_g(y)); g\in G_0, (x,y)\in D_\eps}
$$
First, this set is measurable as a countable union of measurable sets. Next, $e\in G_0$ and therefore $D_\eps\subseteq O_\eps$. In addition, it is clearly $G_0$-invariant by definition, and it is also $G$-invariant, by continuity in measure of the action of $T\times T$ on $L_2(X\times X,\B\times \B,m\times P)$. Finally, note that since $T_gT_h=T_{gh}$, then for every $(x,y)\in O_\eps$ and $g\in G$, there exists $x_0,y_0\in D_\eps$ and $h\in G_0$ such that:
$$
d(T_gx,T_gy) = d(T_gT_hx_0,T_gT_hy_0)=
$$
$$
=d(T_{gh}x_0,T_{gh}y_0)\le \norm{T_{gh}}{op}d(x_0,y_0)<C\cdot\frac\eps{2C}<\eps
$$
Which means $O_\eps\subseteq\bset{(x,y)\in X\times X;~d(x,y)<\eps}$.\\
Denote by $D$ the action invariant metric (such a metric exists by lemma \ref{existenceOfInvMetricLem}). Note that the topology generated by $d$ and the one generated by $D$ are the same. $X$ is a compact metric space, therefore all the metrics defined on it are equivalent, and there exist $M,m>0$ such that 
$$
m\cdot d(x,y)\le D(x,y)\le M\cdot d(x,y)
$$
which means:
$$
B_d\bb{x,\frac {r\cdot m}M}\subseteq B_D\bb{x,r\cdot m}\subseteq B_d(x, r)
$$
Then:
$$
m\times P(D_\eps) = \integrate X{}{P\bb{B_d\bb{x,\frac\eps{2C}}}}m \ge\integrate X{}{P\bb{B_D\bb{x,m\cdot\frac\eps{2C}}}}m
$$
Now, the action is minimal and probability preserving, $D$ is an invariant metric, by lemma \ref{posMeasureBallLem}, for every $x\in X$, $P\bb{B_D\bb{x,m\cdot\frac\eps{2C}}}>0$, and therefore for every $\eps>0$, the integral is positive, which means:
$$
m\times P(O_\eps)\ge m\times P(D_\eps)>0
$$
If we will show that there exists $\eps>0$ such that $m\times P(O_\eps)<1$ then we are done.\\
We saw that $O_\eps\subseteq \bset{(x,y);~d(x,y)<\eps}$. Evidently there exists $\eps>0$ such that-
$$
m\times P\bb{\bset{(x,y);~d(x,y)<\eps}}<1
$$
(Otherwise $m\times P(\bset {(x,x);~x\in X}) = 1$ for some $x\in X$, which is impossible for a product measure if at least one of the measures involved is non-atomic).
\end{proof}
\section{The weak mixing theorem for non-commutative groups}
\begin{defn}[Definition 4.1.3 in \cite{Tomiyama1987}]\label{mooreGroupDfn}
A locally compact group is called a {\bf Moore group} if all irreducible unitary representations are finite dimensional.
\end{defn}
\begin{examp}
According to the spectral theorem, theorem \ref{spectralThm}, every abelian group is a Moore group.
\end{examp}
We will now prove the main theorem:
\begin{thm} \label{grandThm}
Let $G$ be a locally compact Polish Moore group, and let $T:G\rightarrow NST(X)$ be an ergodic action defined on a standard measure space $\sys$. Then the action is weakly mixing if and only if there is no non-trivial finite dimensional, $G$-invariant subspace of $\Lp\infty{\sys}$.
\end{thm}
This theorem gives necessary and sufficient conditions for weak mixing.
\subsection{A non-trivial invariant $\Lp\infty{\sys}$ subspace}
In this subsection, we will show that every action of a locally compact Polish Moore group that has a non-trivial finite dimensional action-invariant $\Lp\infty{\sys}$ subspace, is not weakly mixing.
\begin{prop}\label{existenceBanachKronFactorProp}
Let $T$ be a non-singular ergodic $G$-action defined on $\sys$, a standard $\sigma$-finite measure space. There exists a $G$-invariant, non-trivial, finite dimensional subspace of $\Lp\infty{\sys}$ if and only if there exists a non-trivial Banach-Kronecker factor.
\end{prop}
\begin{proof}
If there exists a non-trivial Banach-Kronecker factor, then there exists a bounded closed set $Y\subseteq \R^n$, a measure $\nu$, a linear action $S$, and a factor map $\pi:X\rightarrow Y$ such that $\nu=m\circ\pi\inv$ and $\pi\circ T=S\circ\pi$. Note that $\pi=(\pi_1,\dots,\pi_n)$, and each of these functions $\pi_j:X\rightarrow \R$ is bounded, since $Y$ is bounded. In addition, 
$$
\pi_j\bb{T_gx}=S_g\pi_j(x)=\sumit i 1 n a_i^j(g)\pi_i(x)
$$
Define $L=span\bset{\pi_j;~1\le j\le n}$, then it is a $G$-invariant non-trivial $\Lp\infty\sys$ subspace.\\
If there exists a finite dimensional non-trivial $G$-invariant subspace of $\Lp\infty{\sys}$, denote it by $V$ and let $S_\infty:=\{f\in \Lp\infty{\sys};\norm f \infty = 1\}$. Let $dim(V) = n$, and let $\bset{f_1,f_2,\dots,f_n}\subseteq S_\infty$ be a basis for $V$. Define the function $F:X\rightarrow \R^n$ by:
$$
F(x) = (f_1(x),f_2(x),\dots,f_n(x))
$$
First of all, $F$ is a measurable function, since for every $1\le j\le n$ the function $f_j$ is a measurable function, and $\C$, the Borel $\sigma$-algebra of $\R^n$, is generated by the product of the Borel $\sigma$-algebra on $\R$.\\
Define the measure $\nu(A)=m(F\inv(A))$. It is well a defined probability measure and $\nu(supp(m))=1$. Define the space $Y =  \overline{supp(\nu)}$.\\
Next, we will show that for every $g\in G$ there exists a matrix $S_g\in M_{n\times n}$ such that $F\circ T_g = S_g\cdot F$. Let there be $g\in G$, then for every $1\le j\le n$ we know that $f_j\circ T_g\in V$, therefore there are $\{a_i^j(g)\}_{i = 1}^n$ such that $\sum_{i = 1}^na_i^j(g)f_i = f_j\circ T_g$. Define the matrix $S_g$ by:
$$
S_g = \begin{pmatrix}
a_1^1(g) & a_2^1(g) & \dots & a_n^1(g)\\ a_1^2(g) & a_2^2(g) & \dots & a_n^2(g)\\ \vdots &\vdots &\dots &\vdots \\ a_1^n(g) & a_2^n(g) & \dots & a_n^n(g)
\end{pmatrix}
$$
Define the action of $G$ on $Y$ by $S_g:Y\rightarrow Y$ by $ S_g(y) = S_g\cdot y$ (since $S_g\in M_{n\times n}$ and $y\in\R^n$ it is well defined).\\
First lets verify that $S_g(y)\in Im(F)$ for $y\in Im(F)$: let $y\in Im(F)$ the there exists $x\in X$ such that $F(x) = y$. Then:
$$
S_g(y) = S_gF(x) =\begin{pmatrix}
a_1^1(g) & a_2^1(g) & \dots & a_n^1(g)\\ a_1^2(g) & a_2^2(g) & \dots & a_n^2(g)\\ \vdots &\vdots &\dots &\vdots \\ a_1^n(g) & a_2^n(g) & \dots & a_n^n(g)
\end{pmatrix}\cdot\begin{pmatrix}
f_1(x)\\f_2(x)\\ \vdots \\ f_n(x)
\end{pmatrix}=
$$
$$
=\begin{pmatrix}
f_1\circ T_g(x)\\f_2\circ T_g(x)\\ \vdots \\ f_n\circ T_g(x)
\end{pmatrix}=F(T_gx)\in Im(F)
$$
We conclude that for every $g\in G$, $S_g(Y)\subseteq Y$ and therefore the action is well defined. Next, by the definition of the measure $\nu$, the transformation $F$ is well defined on a set of full measure, and $F(X)\subseteq Y$ $\nu$-almost everywhere. Finally, $F:X\rightarrow Y$ is a factor map, since $F\circ T_g(x) = F(T_gx) =S_g\cdot F(x) = S_gF(x)$.\\
Moreover, $Y$ is compact- $Y\subseteq[-1,1]^n$, since for ever $1\le j\le n$ the function $f_j\in B_\infty$ and therefore $||f_j||_\infty = 1$, which means it is bounded. In addition, $Y$ is closed by definition, therefore it is compact and specifically Borel-measurable.\\
It is left to show that $\nsys$ is a Banach-Kronecker system. First $\R^n$ is a finite dimensional Banach space, $Y$ a closed bounded subspace, and $G$ is a locally compact Polish group by hypothesis. We will show it is equicontinuous and minimal.\\
{\bf Equicontinuity:} As in the proof of \ref{banachKronThm}, there exists a constant $C\in\R_+$ such that for every $g\in G$, $\norm{S_g}{}<C$. Let there be $\eps>0$, define $\delta=\frac\eps{2C}>0$. Then for every $x,y\in Y$ such that $d(x,y)<\delta$ ($d$ is the canonical metric defined on $\R^n$), for every $g\in G$, $d(S_gx,S_gy)\le\norm{S_g}{}d(x,y)<\eps$.\\
{\bf Minimality:} Denote by $D$ the action invariant metric created in lemma \ref{existenceOfInvMetricLem}. By the definition of $\nu$ there exists $y_0\in Y$ such that for every $\eps>0$, $\nu(B_D(y_0,\eps))>0$. We will show $y_0$ is a transitive point. Let $G_0\subset G$ be a dense countable subgroup. If $y_0$ is not a transitive point, then there exists $\delta>0$ and $y\in Y$ such that for every $g\in G_0$, the intersection $B_D\bb{y,\frac\delta2}\cap B_D\bb{S_g(y_0),\frac\delta 2}=\emptyset$ is empty. Define the set $B_0=\bunions g{G_0} B_D\bb{S_g(y_0),\frac\delta 2}$, it is measurable as a countable union of measurable sets, and for every $g\in G_0,\; gB_0=B_0$, by the definition of a subgroup. By proposition \ref{denseSubgroupErgProp} the action of $G_0$, denoted by $S^0$, is also ergodic. By ergodicity of $S^0$ this set is trivial, but $B_D(y_0,\delta)\subset B_0$ and therefore $\nu(B_0)=1$, which mean $\nu\bb{B_D\bb{y,\frac{\delta}2}}=0$, which is a contradiction to the definition of $Y$. Now, every equicontinuous action has an invariant metric (according to lemma \ref{existenceOfInvMetricLem}), and if a transitive action has an invariant metric, then it is minimal. We conclude the action of $S$ on $\nsys$ is minimal.\\
Finally, the system is not trivial, since the original sub-space was not a trivial one. We conclude this system has a non-trivial Banach-Kronecker factor.
\end{proof}
The following theorem is actually a corollary derived from the proposition above and theorem \ref{banachKronThm}.
\begin{thm}\label{nonMixProp}
Let $G$ be a locally compact Polish group, $\sys$ a standard space, and $T:G\rightarrow NST(X)$ an ergodic action. If there exists a $G$-invariant, non-trivial, finite dimensional subspace of $\Lp\infty{\sys}$, then $T$ is not weakly mixing.
\end{thm}
\begin{proof}
If such a subspace exists by proposition \ref{existenceBanachKronFactorProp}, there exists a non-trivial Banach-Kronecker factor. Now, Banach-Kronecker systems are not weakly mixing according to theorem \ref{banachKronThm}. We conclude our system has a non-trivial factor which is not weakly mixing and therefore this system is not weak mixing.
\end{proof}
\begin{rmk}\label{noNeedFiniteRmk}
Note that in this proof, we did not use the fact that the group is a Moore group. In fact, every non singular action of a locally compact Polish group that has a non-trivial finite dimensional invariant $\Lp\infty{\sys}$ subspace, is not weakly mixing.
\end{rmk}
\subsection{A non weakly mixing action}
In this subsection, we will show that every non-weakly mixing action of a locally compact Polish Moore group has a non-trivial finite dimensional action-invariant $\Lp\infty{\sys}$ subspace.
\subsubsection {Definitions \& Preliminaries: The direct integral}
-\\The following definitions and theorems are quoted from \cite{Dixmier1977}.
\begin{defn}\label{measurableFieldDef}
Let $Z$ be a Borel space, $\mu$ a positive measure, and $\bset{\H(\zeta)}_{\zeta\in Z}$ be a family of Hilbert spaces. A {\bf vector field} is an assignment to each point, of an element of the associated Hilbert space.
\end{defn}
\begin{defn}\label{totalSequenceDef}
Let $\H$ be a Hilbert space. A {\bf total sequence} $\bset{x_n}\subseteq\H$ is a sequence such that $\H$ is the closed linear span of $\bset{x_n}$.
\end{defn}
\begin{defn}\label{measurableFieldHilbertDef}
Let $(Z,\mu)$ be a standard space. A {\bf$\mu$- measurable field of Hilbert spaces over $Z$} is a pair $(\bset{H(\zeta)}_{\zeta\in Z},\Gamma)$, where $\bset{H(\zeta)}_{\zeta\in Z}$ is a family of Hilbert spaces indexed by $Z$, and $\Gamma$ is a set of vector fields satisfying the following conditions:
\begin{enumerate}
\item $\Gamma$ is a vector subspace of $\prods\zeta Z\H(\zeta)$.
\item There exists a sequence $\gamma_1,\gamma_2\dots$ of elements in $\Gamma$, such that for every $\zeta\in Z$ the sequence $\bset{\gamma_n(\zeta)}$ form a total sequence in $\H(\zeta)$.
\item For every $\gamma\in\Gamma$ the function $\zeta\mapsto \norm{\gamma(\zeta)}{}^2$ is $\mu$-measurable.
\item Let $X$ be a vector field. Then if for every $\gamma\in\Gamma$, the function $\binr{X(\zeta)}{\gamma(\zeta)}$ is $\mu$-measurable, then $X\in\Gamma$.
\end{enumerate}
Under these conditions, the elements of $\Gamma$ are called {\bf the measurable vector fields of $(\bset{H(\zeta)}_{\zeta\in Z},\Gamma)$}. If $\gamma_1,\gamma_2\in\Gamma$, then the function $\zeta\mapsto \binr{\gamma_1(\zeta)}{\gamma_2(\zeta)}$ is measurable.
\end{defn}
\begin{defn}\label{measurableFieldOperatorDef}
Let $Z$ be a Borel space, $\mu$ a positive measure, and let $(\bset{H(\zeta)}_{\zeta\in Z},\Gamma)$ be a $\mu$- measurable field of Hilbert spaces over $Z$. An {\bf operator field} is an assignment to each point $\zeta\in Z$, of an element $T(\zeta)\in\mathfrak L(\H(\zeta))$, $\mathfrak L(\H(\zeta))$ the collection of all the linear operators of $\H(\zeta)$. We say that $\zeta\mapsto T(\zeta)$ is a {\bf $\mu$-measurable field of operators} if, for every $\gamma\in\Gamma$, the field $\zeta\mapsto T(\zeta)\gamma(\zeta)$ is measurable. If this is so, the function $\zeta\mapsto\norm{T(\zeta)}{}$ is measurable. Suppose further that this function is essentially bounded, in which case the field is said to be essentially bounded, and $T=\directint T(\zeta)d\mu(\zeta)$ is well defined. The operators of the form $\directint T(\zeta)d\mu(\zeta)$ on $\H$ are said to be {\bf diagonalisable}.
\end{defn}
\begin{defn}\label{measurableFieldRepresentationDef}
Let $Z$ be a Borel space, $\mu$ a positive measure, and $(\bset{H(\zeta)}_{\zeta\in Z},\Gamma)$ a $\mu$- measurable field of Hilbert spaces over $Z$. For each $\zeta\in Z$, let $\pi(\zeta)$ be a representation of a group $G$ in $\H(\zeta)$. Then we call $\zeta\mapsto\pi(\zeta)$ a {\bf field of representations} of $G$. This field is called a {\bf measurable field of representations} if for every $g\in G$ the field of operators $\zeta\mapsto\pi(\zeta)(g)$ is measurable.
\end{defn}
\begin{defn}\label{directIntDef}
Given a measurable field of representations, one could construct for every $g\in G$ the continuous operator $\pi(g)=\directint\pi(\zeta)(g)d\mu(\zeta)$ on the Hilbert space $\H = \directint \H(\zeta)d\mu(\zeta)$. $\pi$ is said to be {\bf the direct integral} of $\bset{\pi(\zeta)}$, and we write $\pi=\directint\pi(\zeta)d\mu(\zeta)$.
\end{defn}
\begin{defn}\label{starAlgDef}
A {\bf *- algebra} is an algebra which is closed with respect to conjugation.
\end{defn}
Note that in the case of operators acting on $\Lp2\sys$ by measure preserving composition (of the form $Uf=f\circ T$, where $T$ is an invertible measure preserving transformation), the conjugate operator is the inverse operator.
\begin{defn}\label{vonNeumannAlgDef}
A {\bf Von-Neumann algebra} (also called {\bf $W^\star$-algebra}) is a $\star$-algebra of bounded operators on a Hilbert space that is closed in the weak operator topology and contains the identity operator.
\end{defn}
\begin{thm}[Mautner, Dixmier]\label{decompositionThm}
Let $Z$ be a Borel space, $\mu$ a positive measure on $Z$, $(\bset{H(\zeta)}_{\zeta\in Z},\Gamma)$ a $\mu$-measurable field of Hilbert spaces over $Z$, $\zeta\mapsto\pi(\zeta)$ a measurable field of representations of $G$ in $\H(\zeta)$. Denote by
$$
\H = \underset Z\directint\H(\zeta)d\mu(\zeta)\; \; \; ;\; \; \; \pi=\underset Z\directint\pi(\zeta)d\mu(\zeta)
$$
and by $\mathfrak Z$ the algebra of diagonalisable operators. Then the following are equivalent:
\begin{enumerate}
\item $\mathfrak Z$ is a maximal commutative Von-Neumann sub-algebra of $\pi(G)'$, where $\pi(G)':=\bset{B;~\forall g\in G.\;B\pi(g)=\pi(g)B, B\text{ bounded }}$.
\item $\pi(\zeta)$ is irreducible for $\mu$-almost every $\zeta\in Z$.
\end{enumerate}
\end{thm}
The theorem above has a rich history. It was originally stated by Mautner, and later proved by many other mathematicians. References can be found in \cite{Dixmier1977}.
\begin{thm}\label{unitaryRepThm}
Let $\H$ be a separable Hilbert space, $\pi$ a representation of $G$ in $\H$, and $\mathcal A$ a maximal commutative Von-Neumann sub-algebra of $\pi(G)'$. Then there exist a standard Borel space $Z$, a bounded positive measure $\mu$ on $Z$, a $\mu$-measurable field $(\bset{H(\zeta)}_{\zeta\in Z},\Gamma)$ of Hilbert spaces over $Z$, a measurable field $\zeta\mapsto\pi(\zeta)$ of irreducible representations of $\mathcal A$ in the $\H(\zeta)$, and an isomorphism of $\H$ onto $\underset Z\directint\H(\zeta)\;d\mu(\zeta)$ (an isomorphism of Hilbert spaces) which transforms $\mathcal A$ into the algebra of diagonalisable operators and $\pi$ into $\underset Z\directint\pi(\zeta)\;d\mu(\zeta)$.
\end{thm}
\begin{rmk}\label{uniquenesRmk}
This decomposition is not always unique. It is known that if the group is postliminal (see definition bellow), then the decomposition is unique, but there are examples where it is not unique. For more details see \cite{Dixmier1977}.
\end{rmk}
\begin{defn}\label{postliminalDef}
A group is called {\bf postliminal} \cite{Dixmier1977} or {\bf GCR} or {\bf type I} \cite{Sakai1967} if for every irreducible representation $\pi$ there exists $g\in G$ such that $\pi(g)$ is compact.
\end{defn}
The definition in \cite{Dixmier1977} is a different one, but in \cite{Sakai1967} they proved the definition above and the definition in \cite{Dixmier1977} are equivalent.
\begin{rmk}
Every Moore group is a type I group.
\end{rmk}
\begin{lem}\label{uniqLem}
Let $f\in\H$ such that $f=0$. Let $(\bset{H(\zeta)}_{\zeta\in Z},\Gamma)$ be a $\mu$- measurable field of Hilbert spaces over $Z$, $\H=\underset Z\directint\H(t)d\lambda(t)$ be a direct integral, $\bar f:Z\rightarrow (\bset{H(\zeta)}_{\zeta\in Z},\Gamma)$ such that $f=\underset Z\directint \bar f(t)\;d\lambda(t)$. Then for $\lambda$- almost every $t\in Z$, $\bar f(t)=0$.
\end{lem}
\begin{proof}
By hypothesis, $0 = \underset Z\directint \bar f(t)d\mu(t)$. If there exists a measurable set $Z_0\subseteq Z$ such that $\lambda(Z_0)>0$ and for every $t\in Z_0$, $\bar f(t)\neq 0$, then $\norm{\bar f(t)}{}>0$.
$$
0=\norm f{}^2=\binr f f=\binr{\underset Z\int\bar f(t)d\mu(t)}{\underset Z\int\bar f(t)d\mu(t)}=
$$
$$
=\underset Z\int\binr{\bar f(t)}{\bar f(t)}d\mu(t)=\underset Z\int\norm{\bar f(t)}{}^2d\lambda(t)\ge\underset {Z_0}\int\norm{\bar f(t)}{0}^2d\lambda(t)>0
$$
which is a contradiction.
\end{proof}
\subsubsection{No Banach-Kronecker factor implies weakly mixing}
\begin{thm}\label{finiteSubThm}
Let $G$ be a locally compact Polish Moore group. Let $T:G\rightarrow NST(X)$ be an ergodic action of $G$ defined on $\sys$ a standard space. If $T$ is not weakly mixing, then there exists a finite dimensional non-trivial $G$-invariant subspace of $\Lp\infty{\sys}$.
\end{thm}
\begin{proof} If $T$ is not weakly mixing, then there exists $\nsys$ a standard probability space, and $S:G\rightarrow PPT(Y)$ ergodic, such that $T\times S$ is not ergodic. If $T\times S$ is not ergodic, then there exists $F:X\times Y\rightarrow\bset{0,1}$ non-constant invariant function. Define the function $\phi:X\rightarrow \Lp2{\nsys}_0$ by:
$$
(\phi(x))(y) = F(x,y)-\int_{Y}F(x,y)d\nu(y)
$$
Note that $\phi(x)\in \Lp2{\nsys}_0$ and, as in the proof of 'the ergodic multiplier theorem', $\phi$ is not constant.\\
Denote by $Isom(\Lp2\nsys_0)$ the set of invertible isomorphisms of $\Lp2\nsys_0$ and define $U:G\rightarrow Isom(\Lp2{\nsys})_0$ the operators defined by $U_gh=h\circ S_{g\inv}$. Then, since $F$ is $T\times S$ invariant, it follows that for every $g\in G$ for $m$ almost every $x\in X$:
\begin{equation}\label{dualEq}
\begin{split}
\phi(T_gx) =  F(T_gx,\cdot) =F(T_gx,S_g(S_{-g}(\cdot))) = \\
=F(x,S_{-g}(\cdot)) = \phi(x)\circ S_{-g}=U_g(\phi(x))
\end{split}
\end{equation}
Let $G_0$ be a countable dense subgroup. For every $g\in G_0$ there is a set of full measure $X_g\subset X$ such that for every $x\in X_g$ equality (\ref{dualEq}) holds. Denote by $X_0=\bintersects g {g_0}  X_g$, then it is of full measure and for every $x\in X_0$ and $g\in G_0$ equation (\ref{dualEq}) holds.\\
Using theorem \ref{unitaryRepThm} there exists a $\lambda$- measurable field of Hilbert spaces over $\T$  $(\bset{H(t)}_{t\in\T},\Gamma)$ such that:
$$
\Lp2{\nsys}_0\simeq\underset \T\directint\H(t)d\lambda(t)=\H
$$
where $\lambda$ is a positive probability measure defined on $\T$. Denote by $\psi:\Lp2\nsys_0\rightarrow \H$ the isomorphism between the two Hilbert spaces, and define $\tilde U:G\rightarrow Isom(\H)$ on $\H$ by $\tilde U_g(h)=\psi(U_g(\psi\inv(h)))$. For every $h\in \H$ there exists $\bar h:\T\rightarrow(\bset{H(t)}_{t\in\T},\Gamma)$, such that for every $t\in\T$, $\bar h(t)\in\H(t)$ and:
$$
h=\underset\T\directint\bar h(t)\;d\lambda(t)
$$
In addition, by the same theorem (\ref{unitaryRepThm}) for every $g\in G$, the operator $\tilde U_g$ is diagonalisable. For every $x\in X_0$ by lemma \ref{uniqLem} there exists a set of full measure $\T_x\subset\T$ such that for every $t\in\T_x$ and $g\in G_0$:
\begin{equation}\label{decompEq}
(\psi(\phi(T_gx)))(t):=(\tilde\phi(T_gx))(t)=(\psi(\phi(x))\circ S_{-g})(t)=\tilde U_g(t)(\tilde\phi(x))(t)
\end{equation}
Note that the above equation is well defined, since $\tilde U_g$ is diagonalisable. By Fubini's theorem, for $\lambda$-almost every $t\in\T$ there exists a set of full measure $X_t\subseteq X_0$ such that equation (\ref{decompEq}) holds for $t$ and for every $x\in X_t$.\\
There exists a set of positive measure $\T_0\subseteq\T$ such that for every $t\in\T_0$, the function $(\tilde\phi(\cdot))(t)$ is not constant as a function of $X$. If not, for almost every $t\in\T$ fix $x_0(t)\in X_0$ such that $(\tilde\phi(x_0(t)))(t)=(\tilde\phi(x))(t)$ for $m$ almost every $x\in X$ and specifically for every $g\in G_0$, $(\tilde\phi(T_gx_0(t)))(t)=(\tilde\phi(x_0(t)))$. Define the function:
$$
E=\underset\T\directint(\tilde\phi(x_0(t)))(t)\;d\lambda(t)
$$
By definition $E\in \H$, and $\psi\inv(E)\in\Lp2\nsys_0$. In addition, for every $g\in G_0$:
$$
\tilde U_gE=\underset\T\directint\tilde U_g(t)(\tilde\phi(x_0(t)))(t)\;d\lambda(t)\overset{(\star)}=
$$
$$
=\underset\T\directint(\tilde\phi(T_gx_0(t)))(t)\;d\lambda(t)=\underset\T\directint(\tilde\phi(x_0(t))(t))\;d\lambda(t)=E
$$
where $(\star)$ is because $x_0\in X_0$, then there is a set of full measure $\T_{x_0}$ such that equation (\ref{decompEq}) holds for every $t\in \T_{x_0}$ and $g\in G_0$.\\
Now, by definition, $\psi\inv(\tilde U_gE)=U_g\psi\inv(E)=\psi\inv(E)\circ S_{g\inv}$, therefore we get that $\psi\inv(E)$ is $G_0$ invariant. Now, $S$ is ergodic, by proposition \ref{denseSubgroupErgProp} the action of $G_0$ on $\nsys$ is also ergodic. $\psi\inv(E)$ is a $G_0$-invariant function, therefore it is constant, but then the original function $F$, is also constant, which is a contradiction.\\
We conclude there exists a set of positive measure $\T_0\subseteq\T$ such that $(\tilde\phi(\cdot))(t)$ is not constant for every $t\in\T_0$. Finally, for almost every $t\in\T$, $\dim(\H(t))<\infty$ and $\tilde U_g(t)$ is unitary (according to the lemma \ref{uniqLem}), specifically there exists $t_0\in\T_0$ such that $\dim(\H(t_0))<\infty$, $\tilde U_g(t_0)$ is a unitary operator, and $(\tilde\phi(\cdot))(t_0)$ is not constant.\\
We will show that $\norm{\tilde\phi(\cdot)(t_0)}{}$ is bounded $m$ almost everywhere.\\
First, $T$ is ergodic, by proposition \ref{denseSubgroupErgProp} the action of $G_0$ on $\sys$ is also ergodic. The function $\norm{(\tilde\phi(\cdot))(t_0)}{}$ is $G_0$ invariant, since for every $g\in G_0$ we know that $(\tilde\phi(T_gx))(t_0)=\tilde U_g(t_0)(\tilde\phi(x))(t_0)$, and, $\tilde U_g(t_0)$ is a unitary operator, then:
$$
\norm{(\tilde\phi(T_gx))(t_0)}{}=\norm{\tilde U_g(t_0)(\tilde\phi(x))(t_0)}{}=\norm{(\tilde\phi(x))(t_0)}{}
$$
By ergodicity of $G_0$, $\norm{(\tilde\phi(\cdot))(t_0)}{}$ is constant. Denote $\norm{(\tilde\phi(\cdot))(t_0)}{}=C_0$ for some $C_0\in\R_+$, then specifically $\norm{(\tilde\phi(\cdot))(t_0)}{}\in\Lp\infty\sys$.\\
Define the function $\varphi:X\rightarrow\H(t_0)$ by $\varphi(x)=(\tilde\phi(x))(t_0)$. It is well defined. Define the space $\tilde V:=span\bset{\tilde U_g(t_0)(\varphi(\cdot));\;g\in G_0}$, then it is non trivial and $G_0$ invariant. We will show it is finite dimensional. Note that since $\dim(\H(t_0))=d<\infty$, then the operators $U_g$ are actually defined uniquely by matrices based on an orthonormal basis, $\bset{e_j}_{i=1}^d$. Denote by $\varphi_{j,k}(x)=\binr{e_j}{\varphi(x)}e_k$.\\
We will show $\tilde V\subseteq span\bset{\varphi_{j,k};~1\le j,k\le d}$. By definition $\varphi=\sumit j 1 d \binr{e_j}{\varphi(x)}e_j\in span\bset{\varphi_{j,k};~1\le j,k\le d}$. Let there be $g\in G_0$ then
$$
\varphi\circ T_g(x)=\tilde U_g(t_0)\varphi(x)=\tilde U_g(t_0)\bb{\sumit j 1 d \binr{e_j}{\varphi(x)}e_j}=
$$
$$
=\sumit j 1 d \binr{e_j}{\varphi(x)}\tilde U_g(t_0)e_j=\sumit j 1 d \binr{e_j}{\varphi(x)}\sumit k 1 da_k^je_k=
$$
$$
=\sumit j 1 d\sumit k 1 d a_k^j\binr{e_j}{\varphi(x)}e_k=\sumit j 1 d\sumit k 1 d a_k^j\varphi_{j,k}(x)\in span\bset{\varphi_{j,k};~1\le j,k\le d}
$$
Specifically $\tilde V$ is finite dimensional. Denote by $\bset{\varphi_1(x),\cdots,\varphi_n(x)}$ an orthonormal basis for $\tilde V$.\\
Next, we know there exists an isomorphism between $\H(t_0)$ and $\R^n$, denote it by $\tilde\pi$. In addition, since $\tilde V$ is not trivial, there exists $j$ such that for some $1\le k\le n$, the function $\tilde\pi(\varphi_k(x))_j$ is not constant. Define $\pi:\H(t_0)\rightarrow\R$ by $\pi(h)=(\tilde\pi(h))_j$.\\
Define the space $V$ by $V=span\bset{\pi(f(\cdot));~f\in\tilde V}$.
\begin{enumerate}
\item $V$ is not trivial: the function $\tilde\pi(\varphi_k(x))_j = \pi(\varphi_k(x))\in V$, and it is a non-constant function.
\item $V$ is finite dimensional: we will show that $\bset{\pi(\varphi_1),\dots,\pi(\varphi_n)}$ is a basis for $V$. Let $f\in V$, then $f(x)=\pi(v(x))=(\tilde\pi(v(x)))_j$ for some $v\in\tilde V$. Since $\bset{\varphi_k(x)}_{k=1}^n$ is a basis for $\tilde V$, there exists $\bset{a_k}$ such that $v=\sumit k 1 n a_k\varphi_k(x)$.
$$
f(x)=(\tilde\pi(v(x)))_j=\bb{\tilde\pi\bb{\sumit k 1 n a_k\varphi_k(x)}}_j=
$$
$$
=\sumit k 1 n a_k(\tilde\pi(\varphi_k(x)))_j=\sumit k 1 n a_k\pi(\varphi_k(x))
$$
\item $V$ is $G_0$ invariant: Let $f\in V$, $f=\tilde\pi(v)$ and $g\in G_0$ then for some $\tilde v\in \tilde V$:
$$
f(T_gx)=\pi\bb{v\bb{T_gx}}=\pi\bb{\tilde U_g\bb{t_0}v(x)}=\pi\bb{\tilde v(x)}\in V
$$
since $\tilde V$ is $G_0$-invariant by definition.
\item $V\subseteq\Lp\infty\sys$: Let $\pi(v)=f\in V$ then:
$$
\abs{f}=\abs{\pi(v)}=\abs{(\tilde\pi(v))_j}\le\norm{\tilde\pi(v)}{}^2=\norm{v}{}^2\le C_0^2
$$
since we saw for every $v\in V, \norm v{}<C_0$.
\end{enumerate}
Finally, we will show $V$ is $G$-invariant.\\
For every $g\in G$ we want to show that $\pi(\tilde\phi(T_gx)(t_0))\in V$, for which it is enough to show $\tilde\phi(T_gx)(t_0)\in\tilde V$. Let there be $g\in G$ and let $\bset {g_n}\subseteq G_0$ be a sequence such that $g_n\underset{n\rightarrow\infty}\longrightarrow g$, then $\norm{\tilde\phi\circ T_{g_n}}{}$ converges in measure to $\norm{\tilde\phi\circ T_g}{}$.  There exists a sub-sequence $\bset{n_k}$ such that $\norm{\tilde\phi\circ T_{g_{n_k}}}{}\overset {a.e}\rightarrow \norm{\tilde\phi\circ T_g}{}$. Since $\tilde V$ is a closed subset in the set of bounded function from $X$ to $\H(t_0)$, if we will show that $\tilde\phi(T_gx)(t_0)$ is bounded, then $\tilde\phi(T_gx)(t_0)\in\tilde V$. We have a uniform bound on $\norm{\tilde\phi(T_{g_n}x)(t_0)}{}= C_0$:
$$
\norm{\tilde\phi(T_gx)(t_0)}{}\overset{a.e}= \limit k\infty\norm{\tilde\phi(T_{g_{n_k}}x)(t_0)}{}=C_0
$$
We conclude $V$ is $G$-invariant.
\end {proof}
\begin{rmk}
While proving theorem \ref{grandThm} we actually proved that an action is weakly mixing if and only if it has a non trivial Banach-Kronecker factor. Moreover, we proved that if an action is weakly mixing, then it has a special non-trivial Banach-Kronecker factor, one that is a unitary action on a Euclidean sphere.
\end{rmk}
\subsection{Corollaries}
\subsubsection{Either finite dimensional or mildly mixing}
\begin{defn}\label{mildMixDef}
A non-singular action of $G$ on a standard space is called {\bf mildly mixing} if for any non-singular properly ergodic action of $G$, the product action is also ergodic.
\end{defn}
\begin{defn}\label{rigidSetDef}
A {\bf rigid set} is a set $B\in\B$ such that there exists a sequence $g_n\rightarrow\infty$ such that:
$$
m(B\symdif g_nB)\underset{n\rightarrow\infty}\longrightarrow 0
$$
Where $g_n\rightarrow\infty$ means that for every compact set $K\subset G$ there exists $N$ such that for every $n>N,\; g_n\nin K$.\\
An action is said to {\bf have no rigid factor} if no non-trivial rigid sets exist.
\end{defn}
\begin{thm}[Schmidt \& Walters \cite{Schmidt1982}]\label{mildRigidThm}
Let $G$ be a locally compact second countable group. Let $T:G\rightarrow NST(X)$ be a properly ergodic action on a standard probability space. This action is mildly mixing if and only if it has no rigid factors.
\end{thm}
\begin{defn}\label{mildMixRepDef}
A representation of $G$ on a Hilbert space, $\H$, is called {\bf mildly mixing} if it has no rigid factor, meaning for every $0\neq h\in\H$ and $g_n\rightarrow\infty$ we have 
$$
\liminf\norm{\pi(g_n)h-h}{}>0
$$
\end{defn}
\begin{defn}
Let $G$ be a locally compact Polish group. A representation $\pi$ is called {\bf mixing} (or {\bf strongly mixing}) if for every $h_1,h_2\in\H$ the function $g\mapsto\binr{\pi(g)h_1}{h_2}$ is in $C_0(G)$, meaning it is continuous and vanishes at infinity.
\end{defn}
\begin{rmk}
Let $G$ be a locally compact Polish group, and let $T:G\rightarrow PPT(X)$ be an action. The representation induced by $T$ is mildly mixing (mixing) if and only if the action is. It is not difficult to verify that if a representation is strongly mixing, then it is mildly mixing.
\end{rmk}
\begin{thm}\label{corollaryThm}
Let $G$ be a locally compact Polish group such that any irreducible unitary representation of $G$ is either finite dimensional or mildly mixing. Let $T:G\rightarrow NST(X)$ be a properly ergodic action of $G$ defined on $\sys$ a standard probability space. Then the action is weakly mixing if and only if every finite dimensional, $G$-invariant subspace of $\Lp\infty{\sys}$ is trivial.
\end{thm}
\begin{proof} Note that by remark \ref{noNeedFiniteRmk}, if there exists a finite dimensional invariant $\Lp\infty{\sys}$ subspace, then the action is not weakly mixing. We will show that if the action is not weakly mixing, then there exists a finite dimensional invariant $\Lp\infty{\sys}$ subspace.\\
If $T$ is not weakly mixing, then there exists $\nsys$ a standard probability space, and $S:G\rightarrow PPT(Y)$ ergodic, such that $T\times S$ is not ergodic. If $T\times S$ is not ergodic, then there exists $F:X\times Y\rightarrow\bset{0,1}$ non-constant invariant function. Define the function $\phi$ and the operator $U$ as in the proof of theorem \ref{finiteSubThm}. Then, since $F$ is $T\times S$ invariant, it follows that for every $g\in G$ for $m$ almost every $x\in X$ equation (\ref{dualEq}) holds.\\
Let $G_0\subseteq G$ be a countable dense subgroup. For every $g\in G_0$ there is a set of full measure $X_g\subset X$ such that for every $x\in X_g$ equality (\ref{dualEq}) holds. Denote by $X_0=\bintersects g {G_0} X_g$. Using theorem \ref{unitaryRepThm}
$$
\Lp2{\nsys}_0\simeq\underset \T\directint\H(t)d\lambda(t)
$$
Using the same notations as in the proof of \ref{finiteSubThm}, there exists a set of positive measure $\T_0\subseteq\T$ such that for every $t\in\T_0$, the function $(\tilde\phi(\cdot))(t)$ is not constant as a function of $X$. If there exists $t_0\in\T_0$ such that $\dim(\H(t_0))<\infty$, then as in the proof of theorem \ref{finiteSubThm}, we have a non-trivial finite dimensional invariant subspace in $\Lp\infty\sys$.\\
Otherwise, for every $t\in\T_0$ we have $\dim(\H(t_0))=\infty$. For every such $t\in\T_0$, by hypothesis, the representation of $G$ on $\H(t)$ is mildly mixing, which means every rigid factor is trivial.\\
We will show that there exists a rigid factor to get a contradiction. Since $T_0$ is an ergodic action, it is recurrent and therefore for almost every $x\in X$ there exists a sequence $\bset{g_n}=\bset{g_n(x)}\subset G_0$ such that for $T_{g_n}x\rightarrow x$. Specifically there exists $x$ such that $\tilde\phi(x)(t_0)\neq \tilde\phi\bb{T_{g_n}x}(t_0)$, but $\norm{\tilde U_{g_n}\tilde\phi(x)(t_0)-\tilde\phi(x)(t_0)}{}=\norm{\tilde\phi\bb{T_{g_n}x}(t_0)-\tilde\phi(x)(t_0)}{}\rightarrow0$, but since the representation is mildly mixing, $\tilde\phi(x)(t_0)=0$, which is a contradiction. We conclude that necessarily there exists $t_0\in\T_0$ such that $\dim(\H(t_0))<\infty$, and therefore there exists a finite dimensional non-trivial $G$-invariant $\Lp\infty\sys$ subspace.
\end{proof}
\begin{prop}\label{mwkIsRGuyProp}
Let $G$ be a locally compact Polish group. If every weakly mixing representation is mildly mixing, then every irreducible representation of $G$ is either finite dimensional or mildly mixing.
\end{prop}
\begin{proof}
Let $\pi$ be an irreducible representation. If $\pi$ is not weakly mixing, then according to theorem \ref{bergelsonRosenblattThm}, there exists a finite dimensional sub-representation, but this representation is irreducible, we conclude it is finite dimensional. Otherwise, $\pi$ is weakly mixing, and by hypothesis mildly mixing.
\end{proof}
\begin{defn}\label{mwm}
A group $G$ is called {\bf minimally weakly mixing} (m.w.m) if $\overline{B(G)}=AP(G)\bigoplus C_0(G)$, where:
$$
B(G):=span\bset{\binr{\pi(*)f}h;~f,h\in\H, \pi \text{ a continuous irreducible representation}}
$$
It is called {\bf minimally weakly almost periodic} if $\overline{B(G)}=C\bigoplus C_0(G)$, where $C$ denotes the constant functions.
\end{defn}
\begin{rmk}\label{mwmExRmk}
Examples of such groups are discussed in \cite{Chou1980}.
\end{rmk}
\begin{thm}[Bergelson \& Rosenblatt \cite{Bergelson1988}]\label{weakIsStrongThm}
If $G$ is minimally weakly mixing, then any weakly mixing representation is strongly mixing.
\end{thm}
\begin{examp}
According to theorem \ref{weakIsStrongThm} and proposition \ref{mwkIsRGuyProp}, every minimally weakly mixing group $G$ will satisfy the conditions of theorem \ref{corollaryThm}. We conclude that theorem \ref{corollaryThm} can be applied to every minimally weakly mixing group $G$.
\end{examp}
\subsubsection{Poisson actions}
\begin{defn}
A quadruple $(X,\B,m,T)$ is called a {\bf non-singular endomorphism} if $\sys$ is a standard probability space, and $T:X_0\rightarrow X_0$ is a measurable non-singular transformation of $X_0\in\B$ such that $m(X_0) = 1$. Note that $T$ is not necessarily invertible.\\
Given a standard space $\sys$ we will denote by $End(X)$ the collection of non-singular endomorphisms of $X$.
\end{defn}
\begin{defn}\label{invFactorDef}
Let $(X,\B,m,T)$ be a non-singular endomorphism, $G$ a locally compact Polish group, $f:X\rightarrow G$ a measurable function, and $\mathbb P\sim m\times m_G$, where  $m_G$ is the Haar measure defined on $G$. Denote by $T_f$ the transformation defined by $T_f(x,g)=(Tx,f(x)g)$. {\bf The invariant factor} of $(X\times G,\B(X\times G),\mathbb P,T_f)$, is a standard probability space $(\Omega,\mathcal F,P)$ equipped with a measurable map $\pi:X\times G\rightarrow\Omega$ such that:
\begin{enumerate}
\item $\mathbb P\circ\pi\inv=P$.
\item $\pi\circ T_f=\pi$.
\item $\pi\inv(\mathcal F)=\bset{A\in\B(X\times G);~T_f\inv A=A}$.
\end{enumerate}
Denote by $Q:G\rightarrow AUT(X\times G)$ the transformation defined by $Q_g(x,h)=(x,hg\inv)$, then $Q_g\circ T_f=T_f\circ Q_g$ and therefore there exists a $P$ non-singular endomorphism $\rho:\Omega\rightarrow\Omega$ such that $\pi\circ Q=\rho\circ\pi$.\\
The non-singular action $(\Omega,\mathcal F,P,\rho)$ is called {\bf the Poisson $G$-action associated to $(T,f)$} and denoted by $\rho(T,f)$.
\end{defn}
\begin{thm}[Aaronson \& Lemanczyk \cite{Aaronson2005}]\label{PoissonThm}
Let $G$ be a locally compact Polish group. If $p\in\mathcal P(G)$ is globally supported, then $\rho(T,f)$ is weakly mixing.
\end{thm}
\begin{cor}\label{PoissonCor}
Let $G$ be a locally compact Polish group. Then its Poisson action has no finite dimensional $L_\infty$ invariant subspaces.
\end{cor}
\begin{proof}
Let $\rho(G,p)$ be a Poisson action of a locally compact group $G$. By theorem \ref{PoissonThm}, it is weakly mixing, and according to theorem \ref{grandThm} it had no finite dimensional $L_\infty$ subspaces.
\end{proof}
\subsection{Examples}
In this section we will present some characterizations of Moore groups. In addition, we will see some examples.
\begin{defn}
A group $G$ is called a {\bf $Z$ group} if the group $G/Z(G)$ is compact, where $Z(G)$ is the centre of $G$.
\end{defn}
\begin{thm}[Grosser \& Moskowitz \cite{Grosser1967} theorem 2.1]\label{ZgroupThm}
Let $G$ be a $Z$-group, then any irreducible unitary representation is finite dimensional.
\end{thm}
\begin{defn}\label{lieDef}
Let $G$ be a connected Lie group, then $G$ is called a {\bf Lie group by definition}. Let $G$ be a group, which is not connected. We will call $G$ a {\bf Lie group by definition} if its component of identity $G_0$ is open in $G$, and if $G_0$ is a Lie group by definition.
\end{defn}
\begin{thm}[Moore \cite{Moore1972}]\label{LieThm}
Let $G$ be a Lie group, then every irreducible unitary representation of $G$ is finite dimensional if and only if there exists an open subgroup of finite index $H$, which is a $Z$-group.
\end{thm}
\begin{defn}\label{projLimDef}
We say $G$ is a {\bf projective limit} and denote it $G = projlim(G_\alpha)$ if there is a family of normal subgroups $\bset{H_\alpha}_{\alpha\in I}$ directed by inclusion such that $G_\alpha = G/H_\alpha$ and $\underset{\alpha\in I}\bigcap H_\alpha=\bset e$, where $H_\alpha$ is compact for each $\alpha\in I$.
\end{defn}
\begin{thm}[Moore \cite{Moore1972}]
For a locally compact group $G$, every irreducible unitary representation is finite dimensional if and only if $G=projlim G_\alpha$ where each $G_\alpha$ is a Lie group, which satisfies that every irreducible unitary representation is finite dimensional.
\end{thm}
\begin{examp}
Every finite group is a Moore group, therefore for every finite group action, theorem \ref{grandThm} holds.
\end{examp}
\begin{examp}
Every abelian group is a Moore group, therefore for every abelian group action, theorem \ref{grandThm} holds.
\end{examp}
\begin{thm}[Peter Weyl theorem]\label{peterWeylThm}
Let $G$ be a compact group, then for every unitary representation $\pi$ there exists a decomposition into finite-dimensional irreducible representations.
\end{thm}
\begin{examp}
By Peter-Weyl theorem, every compact group is a Moore group, therefore for every compact group action, theorem \ref{grandThm} holds.
\end{examp}
\begin{examp}[The infinite Dihedral group]
One of the many representation of this groups is the group generated by a reflection and a rotation. The reflection of $\R^2$ is a multiplication by the matrix 
$$r=
\begin{pmatrix}
			 -1 & 0\\
			 0 & 1
\end{pmatrix}
$$
, and the rotation is a multiplication by
$$s=
\begin{pmatrix}
			 1 & 1\\
			 0 & 1
\end{pmatrix}
$$
Note that 
$$
s^k=
\begin{pmatrix}
			 1 & k\\
			 0 & 1
\end{pmatrix}
$$where $k\in\Z$. This group is clearly infinite.\\
First we will show this group is a Moore group. Let $H=\left<s\right>$, then it is an abelian group. This sub-group is of index 2, since 
$$
s^k\cdot r=
\begin{pmatrix}
			 1 & k\\
			 0 & 1
\end{pmatrix}\cdot \begin{pmatrix}
			 -1 & 0\\
			 0 & 1
\end{pmatrix}
=\begin{pmatrix}
			 -1 & k\\
			 0 & 1
\end{pmatrix}
$$
$$
=\begin{pmatrix}
			 -1 & 0\\
			 0 & 1
\end{pmatrix}\cdot
\begin{pmatrix}
			 1 & -k\\
			 0 & 1
\end{pmatrix}=r\cdot s^{-k}
$$
In addition, if the infinite Dihedral group is endowed with the discrete topology, then $H$ is an open subgroup. By theorem \ref{LieThm}, $D_\infty$ is a Moore group.\\
Let us look at the action of $D_\infty$ on $(\R^2,\B,Leb)$. It can be defined as matrix multiplications, which is a non-singular action on $(\mathbb \R^2,\B,m)$, for some finite measure $m\sim Leb$.\\
Define the function $f:\R^2\rightarrow\R$ by 
$$
f(x,y)=\begin{cases}
		y & \abs y<1\\
		1 & \text{ otherwise }
	\end{cases}
$$
Then it is $D_\infty$ invariant since
$$
f\bb{s\begin{pmatrix}x\\y\end{pmatrix}}=f\bb{\begin{pmatrix}x+y\\y\end{pmatrix}}=f(x,y)
$$
$$
f\bb{r\begin{pmatrix}x\\y\end{pmatrix}}=f\bb{\begin{pmatrix}-x\\y\end{pmatrix}}=f(x,y)
$$
In addition $f\in\Lp\infty{\sys}$ therefore the space $L = span\bset f$ is finite dimensional, $D_\infty$ invariant and non-trivial. By theorem \ref{grandThm} the action in not weakly mixing.
\end{examp}
\begin{examp}[The quaternions group]
Denote the following vectors in $\R^4$:
$$
\indic{}:=	\begin{pmatrix}
		1\\0\\0\\0
		\end{pmatrix}\;\;\;
i:=	\begin{pmatrix}
		0\\1\\0\\0
		\end{pmatrix}\;\;\;
j:=	\begin{pmatrix}
		0\\0\\1\\0
		\end{pmatrix}\;\;\;
k:=	\begin{pmatrix}
		0\\0\\0\\1
		\end{pmatrix}\;\;\;
$$
Define the quaternion multiplication by:
$$
i^2=j^2=k^2=ijk=(-1)
$$
This multiplication can be extended to $H:=span\bset{\indic{},i,j,k}$.\\
Denote by $H^\star:=\bset{\begin{pmatrix}a\\b\\c\\d\end{pmatrix}\in H;~a^2+b^2+c^2+d^2\neq 0}$. Then $H^\star$ is a Lie group with quaternion multiplication. In addition, one of its representations is as complex $2\times 2$ matrices:
$$
\begin{pmatrix}
		a\\b\\c\\d
		\end{pmatrix}\mapsto
		\begin{pmatrix}
		a+ib & c+id\\-(c-id) & a-ib
		\end{pmatrix}
$$
where $i$ is the imaginary number $i=\sqrt{-1}$. We shall use this representation to show this group is a Moore group. Calculations show that-
$$
Z(H^\star)=\bset{\begin{pmatrix}
		r&0\\0&r
		\end{pmatrix}; r\in\R}
$$
$$
H^\star/Z(H^\star)=\bset{\begin{pmatrix}
		a+ib & c+id\\-(c-id) & a-ib
		\end{pmatrix};~a^2+b^2+c^2+d^2=1}
$$
$H^\star/Z(H^\star)$ is a continuous image of a closed bounded set in $\R^4$ and therefore a compact set. We conclude by theorem \ref{LieThm}, that $H^\star$ is a Moore group, and therefore theorem \ref{grandThm} is valid.\\
\end{examp}
\section{Mild mixing vs. weak mixing}
It is known , by definition, that every mildly mixing action is also weakly mixing. For along time it was not known whether these two concepts have a different meaning for non-commutative groups.\\
In this section we will present an example of a probability preserving action of a non-commutative group, which is not mildly mixing but it is weakly mixing. There were some earlier examples of this, but this example is of different nature. We will need some definitions in order to explain how it is different.
\begin{defn}
A group $G$ is called {\bf locally finite} if for every finite set $\bset {g_1,\dots,g_N}\subseteq G$ the group generated by the set is finite.
\end{defn}
\begin{defn}
{\bf The alternating group}, $A_n$, is the group of even permutations of a finite set $\bset {1,\dots,n}$. The alternating group $A_\infty$ is the group of even permutations of $\N$.
\end{defn}
An earlier example was presented by I.Samet in \cite{samet2009}. His example is of an action of the locally finite group, $A_\infty$. The example we will soon present, is of an infinite Lie group, which is not locally finite, and moreover it is uncountable.\\

\begin{defn}\label{HeisenbergDefn}
{\bf The Heisenberg group} $H_3(\R)$ is the group of $3\times3$ upper triangular matrices of the form
$$
M(a,b,c):=\begin{pmatrix}
1 &a &c\\
0 &1 &b\\
0 &0 &1
\end{pmatrix}
$$
This group's unitary irreducible representations are fully characterized. It is known that every unitary irreducible representation is one of the following representations (for more information see \cite{Kirillov2004}):
\begin{enumerate}
\item One dimensional representation on $\mathbb C$ of the form:
$$
\pi_{\alpha,\beta}(M(a,b,c))z:=e^{i\bb{\alpha a+\beta b}}z
$$
\item Infinite dimensional representation on $L_2(\R)$ of the form:
$$
[\pi_\gamma(M(a,b,c))f]\bb{t}:=e^{i\gamma\bb{c+bt}}f(t+a), \;\;\;\gamma\in\R\setminus\bset 0
$$
\end{enumerate}
\end{defn}

The first example of an action of the Heisenberg group, which is weakly mixing but not mildly mixing, was found by Danilenko in \cite{Danilenko2011}.
\begin{thm}[Danilenko \cite{Danilenko2011}]
There is a rigid weakly mixing rank-one action $T$ of $H_3(\R)$.
\end{thm}

We will construct the Gaussian action associated to the infinite dimensional unitary irreducible representation of the Heisenberg group (see theorem \ref{petersonExtendedThm}) and show it is weakly mixing but not mildly mixing.
\begin{examp}[The Heisenberg group]
We will show there exists an action of the Heisenberg group that is weakly mixing but not mildly mixing.\\

Let us look at the infinite dimensional irreducible representation on $\Lp2{\bb{\R}}$ defined by:
$$
[\pi(M(a,b,c))f]\bb{t}:=e^{i\gamma\bb{c+bt}}f(t+a)
$$
It is irreducible, as we mentioned in \ref{HeisenbergDefn}, and therefore it has no invariant subspaces, specifically no finite dimensional non-trivial sub-representations, and therefore this representation is weakly mixing.

Next, we will show the representation has a rigid factor: define the sequence
$$
g_n:=M\bb{0,0,2\pi n+\frac1n}=\begin{pmatrix}
								1 &0 &2\pi n+\frac1n\\
								0 &1 &0\\
								0 &0 &1
							\end{pmatrix}
$$
This sequence is well defined, and $g_n\underset{\abs n\rightarrow\infty}\longrightarrow\infty$. In addition this is a rigid factor in the sense that for every $f\in L_2(\R)$, and $n\in\Z$:
$$
f(t)\neq\left[\pi\bb{M\bb{0,0,2\pi n+\frac1n}}f\right]\bb{t}=e^{i\bb{2\pi n+\frac1n}+i0t}\cdot f\bb{t+0}\rightarrow f\bb{t}
$$
\begin{rmk}\label{heisenbergNotMWMRmk}
According to theorem \ref{weakIsStrongThm}, the Heisenberg group is not minimally weakly mixing, since the above representation is weakly mixing but not strongly mixing.
\end{rmk}
The rest of this paper is dedicated to show the Gaussian action associated to $\pi$ is weakly mixing, but not mildly mixing.
\begin{defn}\label{gaussProcDef}
A centered stationary stochastic process $\bset{X^g;~g\in G}$ is called {\bf Gaussian} if for every $F=\bset{g_1,\dots,g_n}\subset G$, the joint $n$-dimensional distribution $P_F$ is defined by
$$
P_F\bb{\bintersect j 1 n\bset{X^{g_j}\in C_j}}=a\cdot\underset{C_1\times\cdots\times C_n}\int\exp\bb{-\frac12\binr{Mt}{t}}dt_1\dots dt_n
$$
where $a$ is some normalizing constant, $t=(t_1,\dots,t_n)$ and $M$ is the inverse of the regular covariance matrix $C_{ij}:=\E\bb{X^{g_i}X^{g_j\inv}}$
\end{defn}
\begin{defn}\label{gaussActDef}
Let $\sys$ be a standard probability space. An action $T:G\rightarrow PPT(X)$ is called {\bf a Gaussian dynamical system} if there exists $f\in\Lp2\sys_0$ such that:
\begin{enumerate}
\item The corresponding stochastic process $\bset{X^g:=f\circ T_g;\;g\in G}$ is a Gaussian process.
\item The smallest $\sigma$-algebra such that $f\circ T_g$ are all measurable is $\B$.
\end{enumerate}
If we denote by $\H:= span\bset{X^g;\;g\in G}$, then $\H$ is called a {\bf first chaos}, and it can be shown that
$$
\Lp2\sys=\mathbb C\oplus\underset{n=1}{\overset\infty\bigoplus} H^{\odot n}
$$
Where $H^{\odot n}$ is the symmetric tensor product of order $n$.
\end{defn}
\begin{defn}
Let $\H=\Lp2\sys_0$, $\sys$ a $\sigma$-finite measure space. Denote by $\B^{\odot n}$ the set of symmetric measurable sets in $X^n$, and let $m^{\odot n}=\frac1{n!}m^n$. Then we may identify $\H^{\odot n}=\Lp2\sys^{\odot n}_0$ with $\Lp2{(X^n,\B^{\odot n},m^{\odot n})}_0$, which is the subspace of symmetric functions in $\Lp2{\bb{X^n,\B^n,\frac1{n!}m^n}}_0$. This is by letting:
$$
f_1\odot\ldots\odot f_n=\sums\sigma{\mathfrak C_n}\prodit j 1 n f_j(x_{\sigma(j)})
$$
where $\mathfrak C_n$ is the set of all symmetric permutations in $S_n$.\\
The Hilbert direct sum 
$$
\underset{n=1}{\overset \infty\bigoplus}\H^{\odot n}:=\bset{\bset{f_n}_{n=1}^\infty\;\;;\;\;f_n\in\H^{\odot n}\;\;\text{and}\;\; \sumit n 0 \infty \norm{f_n}{}^2<\infty}
$$
is called {\bf the symmetric Fock space over $\H$}.
\end{defn}
\begin{thm}[Peterson \& Sinclair \cite{Peterson2009}]\label{petersonThm}
Let $G$ be a countable discrete group, let $\pi:G\rightarrow\H$ be a unitary representation of $G$ on a Hilbert space $\H$. Then there exists a Gaussian dynamical system $\sys$ and an action $T:G\rightarrow PPT(X)$, such that $\Lp2\sys_0\simeq\underset{n=1}{\overset \infty\bigoplus}\H^{\odot n}$, and $U_T|_{\H}=\pi$. This action is called {\bf the Gaussian process associated to $\pi$}.
\end{thm}
This construction was also mentioned in \cite{Neveu1968}, \cite{Janson1997}, \cite{Glasner2003}, and \cite{Peterson2009}.\\
Denote by $\pi^\mathfrak C$ the Koopman representation on $\underset{n=1}{\overset \infty\bigoplus}\H^{\odot n}$.
\begin{defn}\label{nearActDef}
Let $G$ be a Polish group and $\sys$ a standard probability space. As in \cite{Glasner}, define a {\bf near-action} (or a {\bf Boolean action}) of $G$ on $\sys$ to be a Borel map $T:G\rightarrow PPT(X)$ such that:
\begin{enumerate}
\item If we denote by $e$ the identity element of $G$, then for $m$ almost every $x\in X$, $T_ex = x$.
\item For every $g,h\in G$, $T_g(T_hx) = T_{gh}(x)$ for $m$- almost every $x\in X$.
\end{enumerate}
\end{defn}
\begin{thm}[Mackey \cite{Mackey1962} \& Ramsay \cite{Ramsay1971}]\label{nearActIsActThm}
Let $G$ be a locally compact Polish group. Then every near action is isomorphic to an action.
\end{thm}
For more details see \cite{Glasner}.
\begin{thm}\label{petersonExtendedThm}
Let $G$ be a locally compact Polish group, let $\pi:G\rightarrow\H$ be a unitary representation of $G$ on a Hilbert space $\H$. Then there exists a Gaussian dynamical system $\sys$ and an action $T:G\rightarrow PPT(X)$, such that $\Lp2\sys_0\simeq\underset{n=1}{\overset \infty\bigoplus}\H^{\odot n}$, and $U_T|_{\H}=\pi$. This $G$-action is called {\bf the Gaussian process associated to $\pi$}.
\end{thm}
\begin{proof}
Let $G_0\subset G$ be a countable dense subgroup. Endow $G_0$ with the discrete topology, then by theorem \ref{petersonThm}, there exists a Gaussian action $T^0:G_0\rightarrow PPT(X)$ on a standard probability space $\sys$, such that $U_{T^0}|_{\H}= \pi|_{G_0}$ ($U_{T^0}$ is the Koopman representation induced by $T^0$), and $\Lp2\sys_0\simeq\underset{n=1}{\overset \infty\bigoplus}\H^{\odot n}$. We will show there exists near action of $G$ on the same space, $T$, such that $T|_{G_0}=T^0$, and then use theorem \ref{nearActIsActThm} to conclude that we have an action isomorphic to this near action.\\
Let $\bset{g_k}$ be a Cauchy sequence, then by definition of a representation, $\bset{\pi\bb{g_k}}$ is also a Cauchy sequence. Let there be $n\in\N$, and let us look at the sequence $\bset{\pi^{\odot n}\bb{g_k}}$, when $\pi^{\odot n}(g)$ is the representation defined on $\H^{\odot n}$. We will show this is a Cauchy sequence as well. For every vector $e_1\otimes e_2\otimes\cdots\otimes e_n\in \H^{\otimes n}$:
$$
\norm{\pi^{\otimes n}\bb{g_m}\bb{e_1\otimes e_2\otimes\cdots\otimes e_n}-\pi^{\otimes n}\bb{g_k}\bb{e_1\otimes e_2\otimes\cdots\otimes e_n}}{}=
$$
$$
=\norm{\pi\bb{g_m}e_1\otimes \pi\bb{g_m}e_2\otimes\cdots\otimes \pi\bb{g_m}e_n-\pi\bb{g_k}e_1\otimes \pi\bb{g_k}e_2\otimes\cdots\otimes \pi\bb{g_k}e_n}{}\le
$$
$$
\le \norm{\pi\bb{g_m}e_1\otimes \pi\bb{g_m}e_2\otimes\cdots\otimes \pi\bb{g_m}e_n-\pi\bb{g_k}e_1\otimes \pi\bb{g_m}e_2\otimes\cdots\otimes \pi\bb{g_m}e_n}{}
$$
$$
+\norm{\pi\bb{g_k}e_1\otimes \pi\bb{g_m}e_2\otimes\cdots\otimes \pi\bb{g_m}e_n-\pi\bb{g_k}e_1\otimes \pi\bb{g_k}e_2\otimes \pi\bb{g_m}e_3\otimes\cdots\otimes \pi\bb{g_m}e_n}{}+\cdots+
$$
$$
+\cdots+\norm{\pi\bb{g_k}e_1\otimes\cdots\otimes \pi\bb{g_k}e_{n-1}\otimes \pi\bb{g_m}e_n-\pi\bb{g_k}e_1\otimes\cdots\otimes \pi\bb{g_k}e_n}{}=
$$
$$
=\norm{\bb{\pi\bb{g_m}-\pi\bb{g_k}}e_1\otimes \pi\bb{g_m}e_2\otimes\cdots\otimes \pi\bb{g_m}e_n}{}+
$$
$$
+\norm{\pi\bb{g_k}e_1\otimes \bb{\pi\bb{g_m}-\pi\bb{g_k}}e_2\otimes\cdots\otimes \pi\bb{g_m}e_n}{}+\cdots+
$$
$$
+\cdots+\norm{\pi\bb{g_k}e_1\otimes \cdots\otimes \pi\bb{g_k}e_{n-1}\otimes\bb{\pi\bb{g_m}-\pi\bb{g_k}}e_n}{}
$$
Now every one of these norms converges to zero, since $\bset{\pi\bb{g_k}}$ is a Cauchy sequence and the norm of each $\pi(g)$ is uniformly bounded by 1. Overall, since $n$ is fixed, the sequence $\bset{\pi^{\otimes n}\bb{g_k}}$ is a Cauchy sequence, specifically $\pi^{\odot n}(g)=\pi^{\otimes n}(g)|_{\H^{\odot n}}$ is a Cauchy sequence. We conclude the induced Koopman representation $U_{T^0}$ is continuous on $G_0$.\\
Let $g\in G$, then there exists a Cauchy sequence $\bset{g_n}\subseteq G_0$ such that $g_n\rightarrow g$. The sequence $\bset{U_{T_{g_n}}}$ is a Cauchy sequence in the group of positive isometries, denote by $U_g$ its limit, which is positive and unitary. We will show $U_g$ does not depend on the choice of the sequence $\bset{g_n}$. Let $\bset{g_n'}\subseteq G_0$ be a different sequence such that $\limit n\infty g_n'=g$. Define a new sequence $h_n$ by:
$$
h_n=		\begin{cases}
		g_k\;\;;\;\; n=2k\\
		g'_k\;\;;\;\; n=2k+1
		\end{cases}
$$
Then this is also a Cauchy sequence in $G_0$, and therefore $\bset{U_{T_{h_n}}}$ is also a Cauchy sequence and has a limit, which is the same limit as the limit of $\bset{U_{T_{g_n}}}$, and of $\bset{U_{T_{g_n'}}}$. We conclude the limit does not depend on the choice of the sequence.\\
By Lamperti's theorem, there exists a non-singular transformation $S$ such that, $U_gf=\sqrt{\frac{d\bb{m\circ S}}{dm}}f\circ S$. Denote by $T_g=S=S(g)$. Since the mapping between positive isometries and non-singular actions is a bijection, if $g\in G_0$, then $T^0_g=T_g$. In addition, since $U_{g_n}\rightarrow U_g$ (in the weak operator topology), then for every $k$, $U_g|_{\H^{\odot k}}=U_g^{\odot k}$, $\frac{d\bb{m\circ T_g}}{dm}=1$, and $f\circ T_{g_n}\overset{m}\rightarrow f\circ T$.\\
We will show that $T_{gh}\overset{a.e}=T_gT_h$- let there be $g,h\in G$, then there exist $\bset{g_n},\bset{h_n}\subseteq G_0$ such that $g_n\rightarrow g, h_n\rightarrow h$. Since $T^0$ is an action, there exists a measurable set of full measure, $X_0$, such that for every $x\in X_0$, and every $n,m\in\N$ we have $T_{g_n}T_{h_m}x=T_{g_nh_m}x$. In addition,  as we showed above, if $g_n\rightarrow g$, then $T_{g_n}\overset{m}\rightarrow T_g$, and specifically there exists a subsequence $\bset{n_k}$ such that $T_{g_{n_k}}\overset{a.e}\rightarrow T_g$. Denote the set $X_1:=\bset{x;~T_{g_{n_k}}x\not\rightarrow T_gx \text{ or } T_{h_{m_j}}x\not\rightarrow T_hx}$, then it is of measure zero, and for every $x\in X_0\setminus X_1$:
$$
T_gT_hx=\limit k\infty \limit j\infty T_{g_{n_k}}T_{h_{m_j}}x=\limit k\infty \limit j\infty T_{g_{n_k}h_{m_j}}x=T_{\limit k\infty \limit j\infty g_{n_k}h_{m_j}}x=T_{gh}x
$$
We conclude $T$ is indeed a near action, and by theorem \ref{nearActIsActThm} there exists action, which is isomorphic to this action.
\end{proof}
\begin{thm}\label{GaussianThm}
Let $G$ be a locally compact Polish group, $\H$ a separable Hilbert space, $\pi:G\rightarrow Isom(\H)$ a unitary representation, and let $(X,\B,m,T)$ be the Gaussian process associated to $\pi$. Then the action $T$ is ergodic if and only if it is weakly mixing if and only if $\pi$ is weakly mixing.
\end{thm}
\begin{proof}
Remember that the action described is measure preserving and therefore ergodicity and weak mixing properties of the action are equivalent to the same properties of the Koopman representation. Note that since $\pi^\mathfrak C\otimes\pi^\mathfrak C=\pi^\mathfrak C$, then the action is ergodic if and only if it is weakly mixing. We will show the action is ergodic if and only if $\pi$ is weakly mixing.\\
If $\pi$ is not weakly mixing, then $\pi\otimes\pi$ is not ergodic, and therefore there exists $\xi\in H^{\otimes 2}$, which is $G$-invariant. Denote by $B_2(\H)$ the collection of Hilbert-Schmidt operators defined on $\H$ (see \cite{Dixmier1977} page 314). There exists an isometric isomorphism $\phi:\H^*\otimes\H\rightarrow B_2(\H)$ defined by:
$$
\phi(f\otimes e)=\binr*f e
$$
It is indeed an isometry, and $h\in\H^*\otimes\H$ is symmetric if and only if $\phi(h)$ is self-adjoint. Denote $\tilde\xi=\phi(\xi)$ and let us look at $\bb{\frac{\tilde\xi+\tilde\xi^*}2}$. It is a self adjoint operator and indeed a Hilbert Schmidt operator. In addition $\phi\inv\bb{\bb{\frac{\tilde\xi+\tilde\xi^*}2}}$ is both symmetric and $G$-invariant. We conclude that $\pi^{\mathfrak C}$ has an invariant vector and therefore it is not ergodic.\\
Conversely, assume $\pi^{\mathfrak C}$ is not ergodic, then there exists $\xi\in\underset{n=1}{\overset \infty\bigoplus}\H^{\odot n}$ which is $G$-invariant. Since $\pi^{\mathfrak C}$ is a sub-representation of $\underset{n=1}{\overset \infty\bigoplus}\pi^{\otimes n}$, then this representation also has an invariant vector. Note that 
$$
\underset{n=1}{\overset \infty\bigoplus}\pi^{\otimes n}=\pi\oplus\underset{n=2}{\overset \infty\bigoplus}\pi^{\otimes n}=\pi\otimes\bb{\indic{}\oplus\underset{n=1}{\overset \infty\bigoplus}\pi^{\otimes n}}
$$
Then the product representation $\pi\otimes\bb{\indic{}\oplus\underset{n=1}{\overset \infty\bigoplus}\pi^{\otimes n}}$ has an invariant vector. Now, $\indic{}\oplus\underset{n=1}{\overset \infty\bigoplus}\pi^{\otimes n}$ is also a unitary representation. We conclude $\pi\otimes\pi$ has an invariant vector, which means it is not ergodic, and $\pi$ is not weakly mixing.
\end{proof}
The first proof of this theorem was done by Maruyama for $G=\Z$ in 1949 (see \cite{Maruyama1949}). The proof above was taken from notes created by Peterson. For more details see 
\href{http://www.math.vanderbilt.edu/~peters10/}{Peterson's homepage}.\\

According to theorem \ref{petersonExtendedThm} applied to the unitary representation of $H_3(\R)$ described above, the action $T$ is weakly mixing. If we will show it is also not mildly mixing, then we will have the desired example.\\
We will show there exists a rigid sequence, and by Schmidt-Walter's result (theorem \ref{mildRigidThm}) the action is not mildly mixing. But since $\pi$ is contained in the Koopman representation of $H_3(\R)$ on $\Lp2{\sys}_0$, then we already saw there is a rigid factor. We conclude this action is weakly mixing and not mildly mixing.
\end{examp}
\section{Acknowledgments}
This work is the author's master thesis.\\
The author would like to personally thank her MSc adviser, Jon Aaronson, for acquainting her to the  world of dynamics, and affording her guidance and patience. Also, to Dror Speiser who drew the authors attention to the the fact that the infinite Dihedral group and the quaternions group are both infinite non-abelian Moore groups. Many thanks are also given to Eli Glasner, Benjamin Weiss, and Alexandre I. Danilenko for finding mistakes in previous versions. Last but not least to Tom Meyerovitch, Zemer Kosloff, \& Michael Bromberg for useful conversations. Their support, and contribution are greatly appreciated.\\
This work has been partially supported by I.S.F grants numbers: 1114/08, and 1157/08.
\bibliographystyle{plain}
\nocite{*}
\bibliography{thesis-references}
\end{document}